\documentclass[english]{article}
\usepackage[T1]{fontenc}
\usepackage[latin9]{inputenc}
\usepackage{float}
\usepackage{amsmath}
\usepackage{amsthm}
\usepackage{amssymb}
\usepackage{stmaryrd}
\usepackage{graphicx}

\makeatletter
\numberwithin{equation}{section}
\numberwithin{figure}{section}
\theoremstyle{plain}
\newtheorem{thm}{\protect\theoremname}
\theoremstyle{definition}
\newtheorem{defn}[thm]{\protect\definitionname}
\theoremstyle{remark}
\newtheorem{rem}[thm]{\protect\remarkname}
\theoremstyle{plain}
\newtheorem{prop}[thm]{\protect\propositionname}
\theoremstyle{plain}
\newtheorem{fact}[thm]{\protect\factname}
\theoremstyle{plain}
\newtheorem{lem}[thm]{\protect\lemmaname}
\theoremstyle{plain}
\newtheorem{cor}[thm]{\protect\corollaryname}
\theoremstyle{remark}
\newtheorem*{rem*}{\protect\remarkname}
\theoremstyle{remark}
\newtheorem{claim}[thm]{\protect\claimname}

\usepackage{fullpage}
\usepackage{hyperref}
\usepackage{centernot}
\usepackage[bottom]{footmisc}
\hypersetup{
  colorlinks   = true, 
  urlcolor     = red, 
  linkcolor    = blue, 
  citecolor   = blue 
}
\date{}
\makeatother

\usepackage{babel}
\providecommand{\claimname}{Claim}
\providecommand{\corollaryname}{Corollary}
\providecommand{\definitionname}{Definition}
\providecommand{\factname}{Fact}
\providecommand{\lemmaname}{Lemma}
\providecommand{\propositionname}{Proposition}
\providecommand{\remarkname}{Remark}
\providecommand{\theoremname}{Theorem}

\begin{document}
\global\long\def\goinf{\rightarrow\infty}%
\global\long\def\gozero{\rightarrow0}%
\global\long\def\bra{\langle}%
\global\long\def\ket{\rangle}%
\global\long\def\union{\cup}%
\global\long\def\intersect{\cap}%
\global\long\def\abs#1{\left|#1\right|}%
\global\long\def\norm#1{\left\Vert #1\right\Vert }%
\global\long\def\floor#1{\left\lfloor #1\right\rfloor }%
\global\long\def\ceil#1{\left\lceil #1\right\rceil }%
\global\long\def\expect{\mathbb{E}}%
\global\long\def\e{\mathbb{E}}%
\global\long\def\r{\mathbb{R}}%
\global\long\def\n{\mathbb{N}}%
\global\long\def\q{\mathbb{Q}}%
\global\long\def\c{\mathbb{C}}%
\global\long\def\z{\mathbb{Z}}%
\global\long\def\grad{\nabla}%
\global\long\def\t{\mathbb{T}}%
\global\long\def\all{\forall}%
\global\long\def\eps{\varepsilon}%
\global\long\def\quadvar#1{V_{2}^{\pi}\left(#1\right)}%
\global\long\def\cal#1{\mathcal{#1}}%
\global\long\def\cross{\times}%
\global\long\def\del{\nabla}%
\global\long\def\parx#1{\frac{\partial#1}{\partial x}}%
\global\long\def\pary#1{\frac{\partial#1}{\partial y}}%
\global\long\def\parz#1{\frac{\partial#1}{\partial z}}%
\global\long\def\part#1{\frac{\partial#1}{\partial t}}%
\global\long\def\partheta#1{\frac{\partial#1}{\partial\theta}}%
\global\long\def\parr#1{\frac{\partial#1}{\partial r}}%
\global\long\def\curl{\nabla\times}%
\global\long\def\rotor{\nabla\times}%
\global\long\def\one{\mathbf{1}}%
\global\long\def\Hom{\text{Hom}}%
\global\long\def\p{\mathbb{P}}%
\global\long\def\almost{\mathbf{\approx}}%
\global\long\def\tr{\text{Tr}}%
\global\long\def\var{\text{Var}}%
\global\long\def\cov{\text{Cov}}%
\global\long\def\onenorm#1{\left\Vert #1\right\Vert _{1}}%
\global\long\def\twonorm#1{\left\Vert #1\right\Vert _{2}}%
\global\long\def\Inj{\mathfrak{Inj}}%
\global\long\def\inj{\mathsf{inj}}%
\global\long\def\Inf{\mathrm{Inf}}%
\global\long\def\i{\mathrm{I}}%
\global\long\def\sign{\mathrm{sign}}%
\global\long\def\tensor{\otimes}%
\global\long\def\duplimatch#1#2{#1\diamond#2}%
\global\long\def\aut{\mathrm{Aut}}%
\global\long\def\twist#1#2#3{#1\stackrel{#3}{\star}#2}%

\global\long\def\g{\mathfrak{\cal G}}%
\global\long\def\f{\mathfrak{\cal F}}%
\global\long\def\circlaw{\mathfrak{\mu_{\text{circ}}}}%
\global\long\def\permvector{\mathfrak{\bar{\sigma}}}%

\title{Randomly twisted hypercubes -- between structure and randomness}
\author{Itai Benjamini\thanks{Weizmann Institute of Science, Department of Mathematics. Email: itai.benjamini@gmail.com.
Supported by the Israel Science Foundation.}, Yotam Dikstein\thanks{Weizmann Institute of Science, Department of Computer Science and
Applied Mathematics. Email: yotam.dikstein@weizmann.ac.il.}, Renan Gross\thanks{Weizmann Institute of Science, Department of Mathematics. Email: renan.gross@weizmann.ac.il.
Supported by the Adams Fellowship Program of the Israel Academy of
Sciences and Humanities, the European Research Council and by the
Israeli Science Foundation.} ~and Maksim Zhukovskii\thanks{Tel Aviv University, School of Mathematical Sciences. Email: zhukmax@gmail.com.
Supported in part by the Israel Science Foundation grant 2110/22.}}
\maketitle
\begin{abstract}
Twisted hypercubes are generalizations of the Boolean hypercube, obtained
by iteratively connecting two instances of a graph by a uniformly
random perfect matching. Dudek et al. showed that when the two instances
are independent, these graphs have optimal diameter. 

We study twisted hypercubes in the setting where the instances can
have general dependence, and also in the particular case where they
are identical. We show that the resultant graph shares properties
with random regular graphs, including small diameter, large vertex
expansion, a semicircle law for its eigenvalues and no non-trivial
automorphisms. However, in contrast to random regular graphs, twisted
hypercubes allow for short routing schemes.

\tableofcontents{}
\end{abstract}

\section{Introduction and construction}

The Boolean hypercube $Q_{n}$ is the graph whose vertex set is $V\left(Q_{n}\right)=\left\{ 0,1\right\} ^{n}$
and whose edge set is $E\left(Q_{n}\right)=\left\{ \left\{ x,y\right\} \mid\text{\ensuremath{x} and \ensuremath{y} differ by exactly one coordinate}\right\} $.
One appeasing property of the hypercube graph is its recursive construction:
starting with $Q_{1}$ as a single edge, $Q_{n}$ is given by the
Cartesian graph product $Q_{n}=Q_{1}\boxempty Q_{n-1}$; essentially,
the Cartesian product with an edge amounts to matching together the
corresponding vertices of two disjoint copies of $Q_{n-1}$. See Figure
\ref{fig:hypercube_construction} for the first steps of this process.

\begin{figure}[H]
\begin{centering}
\includegraphics[scale=0.4]{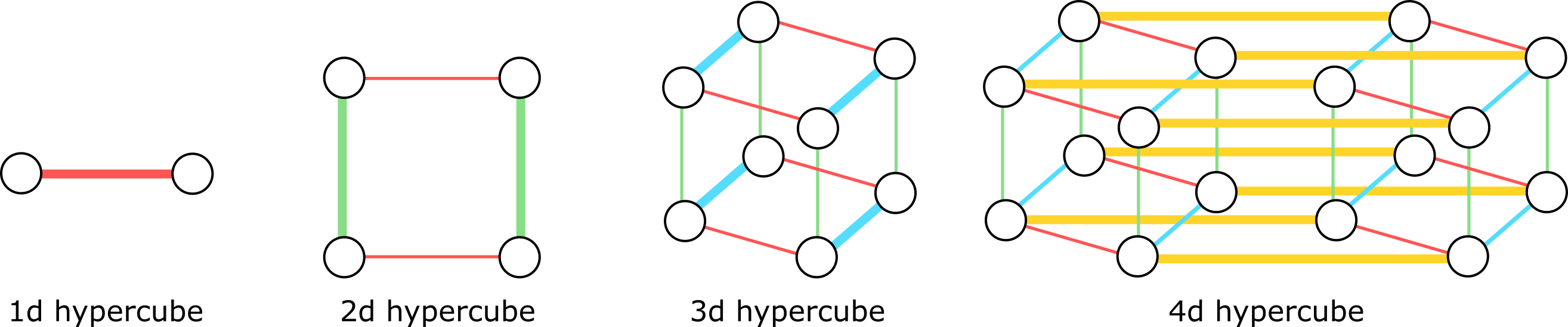}
\par\end{centering}
\centering{}\caption{\label{fig:hypercube_construction}The recursive construction of the
hypercube. The color of an edge indicates in which step it was created.
Newly created edges are in bold. }
\end{figure}
Generalizing this procedure gives rise to the definition of a \emph{twisted
hypercube}, which is obtained by iteratively applying perfect matchings
between the vertices of two copies of the original graph.
\begin{defn}[$\sigma$-twist]
Let $G_{0}=\left(V,E_{0}\right)$ and $G_{1}=\left(V,E_{1}\right)$
be two finite graphs on the same vertex set, and let $\sigma$ be
a permutation of the vertices $V$. The \emph{$\sigma$-twist} operation,
denoted $\twist{G_{0}}{G_{1}}{\sigma}$, produces a graph $\twist{G_{0}}{G_{1}}{\sigma}=\left(V',E'\right)$,
defined as follows. For $i=0,1$, let $V^{i}=\left\{ \left(x,i\right)\mid x\in V\right\} $
and $F^{i}=\left\{ \left\{ \left(x,i\right),\left(y,i\right)\right\} \mid\left\{ x,y\right\} \in E_{i}\right\} $.
Then $\twist{G_{0}}{G_{1}}{\sigma}$ has vertex set
\[
V'=V^{0}\union V^{1}
\]
and edge set 
\[
E'=F^{0}\union F^{1}\union\left\{ \left\{ \left(x,0\right),\left(\sigma\left(x\right),1\right)\right\} \mid x\in V\right\} .
\]
Alternatively, if $A_{i}\in\r^{m\times m}$ is the adjacency matrix
of $G_{i}$, and $P$ is the $m\times m$ permutation matrix representing
$\sigma$, then the adjacency matrix of $\twist{G_{0}}{G_{1}}{\sigma}$
is given by
\begin{equation}
\left(\begin{array}{cc}
A_{0} & P\\
P^{T} & A_{1}
\end{array}\right).\label{eq:matrix_next_step}
\end{equation}
\end{defn}

\begin{defn}[Twisted hypercube]
Using the $\sigma$-twist operation, we define a class of recursively
constructed graphs.
\end{defn}

\begin{enumerate}
\item A\emph{ twisted hypercube} graph of generation $n$, denoted $G_{n}$,
is defined as follows. For $n=0$, $G_{0}$ is an isolated vertex,
labeled by $\emptyset$, and for $n=1$, $G_{1}$ is a single edge,
i.e. $V\left(G_{1}\right)=\left\{ 0,1\right\} $ and $E\left(G_{1}\right)=\left\{ \left\{ 0,1\right\} \right\} $.
For general $n>1$, let $G_{n-1}^{0}$ and $G_{n-1}^{1}$ be two twisted
hypercubes of generation $n-1$, let $\sigma_{n-1}$ be a permutation
on $\left\{ 0,1\right\} ^{n-1}$ vertices, and define
\begin{align}
G_{n} & =\twist{G_{n-1}^{0}}{G_{n-1}^{1}}{\sigma_{n-1}}.\label{eq:twisted_hypercube_definition}
\end{align}
\item It is possible to consider random permutations in this construction;
$G_{n}$ is then called a \emph{random twisted hypercube}. In this
case, the definition requires also specifying the joint distribution
of $G_{n-1}^{0}$ and $G_{n-1}^{1}$ in the $\sigma$-twist operation.
For the rest of the paper, we assume that all the permutations $\sigma_{k}$
are chosen uniformly at random for every $k$.
\item When the all permutations $\sigma_{k}$, $k=1,\ldots,n-1$ are independent
of all other permutations, and the two instances $G_{n-1}^{0}$ and
$G_{n-1}^{1}$ are independent for all $n$, we call $G_{n}$ an \emph{independent
twisted hypercube.}
\item When the two instances $G_{n-1}^{0}$ and $G_{n-1}^{1}$ are identical
for all $n$, we call the graph a \emph{duplicube}. In this case,
the graph $G_{n}$ can be described by a single sequence of permutations
$\permvector=\left(\sigma_{k}\right)_{k=1}^{\infty}$, where each
$\sigma_{k}$ is a permutation on $\left\{ 0,1\right\} ^{k}$: the
vertex set is $V\left(G_{n}\right)=\left\{ 0,1\right\} ^{n}$, and
for every $k\in\left[n\right]$, the vertex $x=\left(x_{1},\ldots x_{k-1},0,x_{k+1},\ldots,x_{n}\right)$
is connected to $y=\left(\sigma_{k-1}\left(x_{1},\ldots,x_{k-1}\right),1,x_{k+1},\ldots,x_{n}\right)$.
We write $G_{n}=G_{n}\left(\bar{\sigma}\right)$ when we wish to stress
the dependence on the permutations. See Figure \ref{fig:duplicube_construction}
for the first steps of this process.
\end{enumerate}
\begin{figure}[H]
\begin{centering}
\includegraphics[scale=0.4]{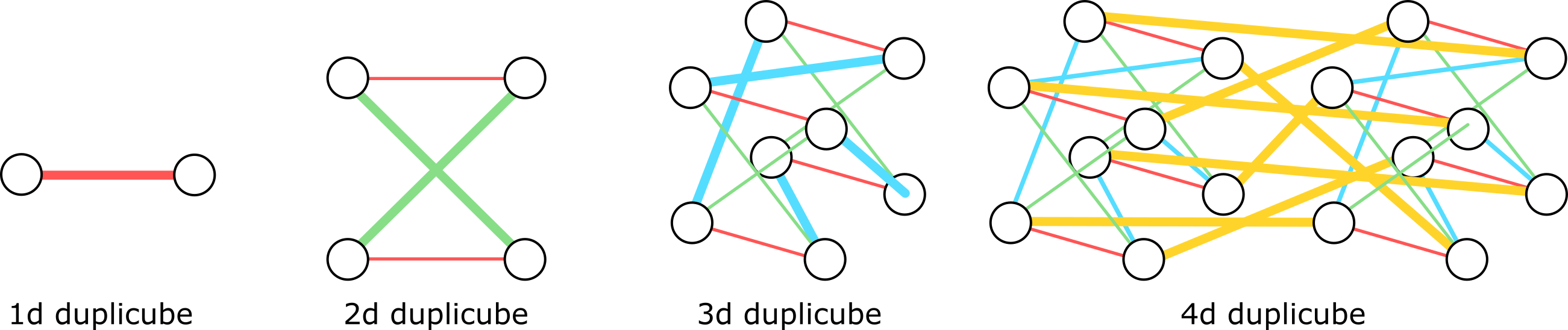}
\par\end{centering}
\centering{}\caption{\label{fig:duplicube_construction}An example of the recursive construction
of the duplicube, using random matchings. The color of an edge indicates
in which step it was created. Newly created edges are in bold.}
\end{figure}

The term \emph{twisted cubes} was first introduced in the context
of routing in computing networks \cite{hilbers_et_al_twisted_cube,abraham_padmanabhan_twisted_cube}.
The idea is that slight modifications to the structure of the hypercube
can yield graphs with both better diameter (and so, smaller latency)
and better connectivity (and so, better fault-tolerance) than the
hypercube. Dudek et al. \cite[definition 2]{dudek_et_al_randomly_twisted_hypercube}
first introduced randomness to these constructions, and studied independent
instances connected by uniform matchings. They named their construction
\emph{random twisted hypercubes}. Since our definition generalizes
theirs by allowing different joint distributions of matched instances,
we have chosen to use the name \emph{random twisted cubes }for the
general case, and \emph{independent twisted cubes }for their special
case. 

Any twisted hypercube $G_{n}$ is an $n$-regular graph with $N=2^{n}$
vertices. When $\sigma_{k}$ is the identity permutation for every
$k$, then $G_{n}$ is just the Boolean hypercube graph $Q_{n}$.
The hypercube has diameter $n$, has poor vertex- and edge-expansion
(relative to the fact that its degree grows with the graph size; see
Section \ref{subsec:vertex_expansion}), and a random-walk mixing
time of order $\Theta\left(n\log n\right)$ \cite{levin_peres_markov_chains_and_mixing_times}.
Many other geometric and structural properties of the hypercube are
known (e.g. distances between vertices \cite{hypercube_survey} and
isoperimetric inequalities for various sets \cite{odonnell_analysis_of_boolean_functions,jiang_yehudayoff_isopetimetric_inequalities,bobkov_isoperimetric_inequality_on_the_discrete_cube}). 

Another well-researched class of $n$-regular graphs are the uniformly
random regular graphs. With probability $1-o\left(1\right)$, a random
$n$-regular graph on $2^{n}$ vertices has diameter $\Theta\left(n/\log n\right)$
\cite{bollobas_dela_vega_diameter_of_random_regular}, has high edge-expansion
\cite{bauerschmidt_huang_knowles_yau_intermediate_degree} and a random-walk
mixing time of order $\Theta\left(n/\log n\right)$ \cite{lubetzky_sly_cutoff_phenomena}.
Further, its eigenvalues follow a semicircle distribution \cite{dumitriu_pal_sparse_regular_random_graphs}. 

For fixed $n$ and $N\to\infty$, the uniform distribution over $n$-regular
graphs on $N$ vertices can be approximated by adding $n$ successive
random perfect matchings on $N$ isolated vertices, where the $i$-th
matching is uniform over all matchings on previously-unmatched pairs
of vertices \cite[Theorem 8]{gao_perfect_matchings_regular_graphs}.
In contrast, consider the random twisted hypercube $\boldsymbol{G}_{n}$,
where all $\sigma_{k}$ are uniformly random permutations on $\left\{ 0,1\right\} ^{k}$,
with independence between different $k$'s . It consists of a union
of $n$ independent matchings as well, but these matchings are not
uniformly random. For example, the last matching is a uniformly random
matching only between the two instances of $\boldsymbol{G}_{n-1}$,
while earlier matchings consist of a union of smaller matchings; in
the case of the duplicube, they consist of \emph{copies} of smaller
matchings and therefore have even stronger dependencies between the
edges. In this sense, the random twisted hypercube $\boldsymbol{G}_{n}$
is a hybrid between the structure of the Boolean hypercube and the
randomness of a random $n$-regular graph. It is therefore natural
to ask how its various geometric and structural properties compare
to those of the hypercube and random $n$-regular graphs. 
\begin{rem}
An $n$-th iteration duplicube is defined by a single sequence of
permutations $\sigma_{1},\ldots,\sigma_{n-1}$. To sample such a sequence,
one requires approximately $\Theta(n2^{n})$ random bits. An independent
twisted hypercube, on the other hand, requires $\Theta(n^{2}2^{n})$
random bits to sample since it is defined using $2^{n-k-1}$ independent
copies of $\sigma_{k}$ for every $k$. As we will see in the next
section, despite the fact that it uses less randomness, the structural
properties of the duplicube still match those of the independent twisted
cube -- it has optimal diameter and constant vertex expansion.
\end{rem}

\section{Our results}

In this work, we study the diameter, expansion, eigenvalues, and symmetries
of a random twisted hypercube $\boldsymbol{G}_{n}$. All our theorems,
except Theorem \ref{thm:small_diameter_duplicube}, hold for any twisted
hypercube where the matchings $\sigma_{k}$ are uniformly random and
with independence between permutations of different generations, regardless
of the joint distribution of the instances in (\ref{eq:twisted_hypercube_definition}). 

\subsection{The diameter}

For a graph $G=\left(V,E\right)$, let $d_{G}:V^{2}\to\r$ be the
graph distance between two vertices. The diameter of a graph is the
maximum distance in the graph, i.e. $D\left(G\right):=\max\left\{ d_{G}\left(x,y\right)\mid x,y\in V\right\} $.
An immediate result shows that the diameter of the hypercube $Q_{n}$
has the worst possible diameter out of all twisted hypercube graphs.
\begin{prop}
\label{prop:trivial_upper_bound}For every choice of permutations
$\sigma_{k}$, we have $D\left(G_{n}\right)\leq D\left(Q_{n}\right)=n$.
\end{prop}

\begin{proof}
By induction. For $n=1$, it is clear. In the general case, let $x,y\in V\left(G_{n}\right)$,
and denote $x=\left(\tilde{x},x_{n}\right)$, $y=\left(\tilde{y},y_{n}\right)$.
If $x_{n}=y_{n}$, then $\tilde{x},\tilde{y}\in V\left(G_{n-1}\right)$,
and $d_{G_{n}}\left(x,y\right)=d_{G_{n-1}}\left(\tilde{x},\tilde{y}\right)\leq n-1$.
Otherwise, $x$ is connected to some $\left(x',1-x_{n}\right)$, and
\[
d_{G_{n}}\left(x,y\right)\leq1+d_{G_{n-1}}\left(x',\tilde{y}\right)\leq n.
\]
\end{proof}
The following lower bound is also immediate.
\begin{prop}
For every choice of permutations $\sigma_{k}$, we have $D\left(G_{n}\right)\geq\left(n-1\right)/\log_{2}n$.
\end{prop}

\begin{proof}
If $G_{n}$ has diameter $d$, then the ball $B\left(v,d\right)$
of radius $d$ around any vertex $v$ must contain the entire graph.
Since the graph is $n$-regular, the number of vertices in this ball
is smaller than $2n^{d}$, and we get
\[
2^{n}=\abs{B\left(v,d\right)}\leq2n^{d},
\]
yielding
\[
d\geq\frac{n-1}{\log_{2}n}.
\]
\end{proof}
It was shown by Dudek et al. \cite{dudek_et_al_randomly_twisted_hypercube}
that for the independent twisted hypercube (where the permutations
$\sigma_{k}$ are chosen uniformly at random, and the instances of
$\boldsymbol{G}_{n-1}$ in the $\sigma$-twist operation $\boldsymbol{G}_{n}=\twist{\boldsymbol{G}_{n-1}^{0}}{\boldsymbol{G}_{n-1}^{1}}{\sigma_{n-1}}$
are independent), the diameter of $\boldsymbol{G}_{n}$ is almost
surely asymptotic to $n/\log_{2}n$. We show that their proof technique
carries over to the duplicube as well.
\begin{thm}
\label{thm:small_diameter_duplicube} Let $\boldsymbol{G}_{n}$ be
the random duplicube. Then $D(\boldsymbol{G}_{n})=\frac{n}{\log_{2}n}+O(\frac{n}{\log^{2}n})$
with probability $\geq1-o(2^{-n})$.
\end{thm}

Moreover, we show that regardless of the joint distribution of the
two instances of $\boldsymbol{G}_{n-1}$, the diameter is asymptotically
better than that of Proposition \ref{prop:trivial_upper_bound} by
at least a $\log\log n/\log\log\log n$ factor. The following theorem
is proved in Section \ref{subsec:postponed_diameter_proofs}.
\begin{thm}
\label{thm:small_diameter_general_case}There exists a constant $C>0$
such that
\begin{equation}
D\left(\boldsymbol{G}_{n}\right)\leq Cn\frac{\log\log\log n}{\log\log n}\label{eq:n_over_log_n_diameter}
\end{equation}
with probability $1-o\left(2^{-n}\right)$.
\end{thm}

\begin{rem}
The proof of Proposition \ref{prop:trivial_upper_bound} also gives
a simple routing scheme between any two vertices $x,y$: when at $x$,
let $k\in\left[n\right]$ be the largest index such that $x_{k}\neq y_{k}$,
and go along the edge created by $\sigma_{k-1}$. Thus, we always
have a local routing scheme which gives a good approximation to the
diameter, as well as the average distance between pairs of vertices.
Contrast this with general random $n$-regular graphs, where there
is no known local easy way to find an approximation to the minimal
path between two vertices.
\end{rem}

\begin{rem}
It might be possible to improve the factor $\frac{\log\log\log n}{\log\log n}$
in Theorem \ref{thm:small_diameter_general_case} by a more careful
analysis of the quantities $\alpha\left(n\right)$ and $\beta\left(n\right)$
that appear in the theorem's proof. We showed that the diameter of
a random twisted hypercube is asymptotically less than that of the
Boolean hypercube, yet we have no intuition to the correct diameter.
\end{rem}

\subsection{Vertex expansion\label{subsec:vertex_expansion}}

Let $G=\left(V,E\right)$ be any graph. For a set $S\subseteq V$,
let $\partial S$ be its set of neighbors, i.e. $\partial S=\left\{ x\notin S\mid\exists y\in S\text{ such that }\left\{ x,y\right\} \in E\right\} $. 
\begin{defn}[Vertex expander]
Let $0<\eta<1$ and $\alpha>0$, and let $G=\left(V,E\right)$ be
a graph. A set $S\subseteq V$ is said to have \emph{$\alpha$-expansion}
if $\abs{\partial S}\geq\alpha\abs S$. The graph $G$ is an $\left(\eta,\alpha\right)$-vertex-expander
if $S$ has $\alpha$-expansion for all $S\subseteq V$ of size $\abs S\leq\eta\abs V$.
\end{defn}

The hypercube $Q_{n}$ is not an $\left(\eta,\alpha\right)$-vertex-expander
for any constants $\eta,\alpha>0$. To see this, fix some constant
$\eta>0$. There is some $\rho>0$ so that the ball $S=\left\{ x\in\left\{ 0,1\right\} ^{n}\mid\sum x_{i}\leq\ceil{n/2-\rho\sqrt{n}}-1\right\} $
has size $\left(\eta+o\left(1\right)\right)2^{n}$. However, its boundary
is $\partial S=\left\{ x\in\left\{ 0,1\right\} ^{n}\mid\sum x_{i}=\ceil{n/2-\rho\sqrt{n}}\right\} $
and has size at most ${n \choose n/2}=2^{n}\left(1+o\left(1\right)\right)/\sqrt{\pi n}$.
Thus $Q_{n}$ cannot have an expansion factor $\alpha$ asymptotically
larger than $\frac{1}{\sqrt{n}}$ for any constant $\eta$. The random
twisted hypercube graph, on the other hand, achieves constant expansion
with high probability. The following theorem is proved in Section
\ref{subsec:postponed_expansion_proofs}.
\begin{thm}
\label{thm:good_vertex_expander}For every $\eta\in(0,1)$ there exists
a constant $\alpha>0$ such that 
\[
\lim_{n\to\infty}\p[\boldsymbol{G}_{n}\text{ is a }(\eta,\alpha)\text{-vertex expander}]=1.
\]
\end{thm}

In fact, the proof of Theorem \ref{thm:good_vertex_expander} shows
that $\p[\boldsymbol{G}_{n}\text{is not a }(\eta,\alpha)\text{-vertex expander}]=O(2^{-cn})$
for some constant $c>0$ that depends on $\eta$.
\begin{rem}
It is also possible to talk about edge expanders, and compare the
size of a set $S$ to the number of edges connecting it to $\partial S$.
Both $Q_{n}$ and $G_{n}$ are not very good edge expanders (for any
choice of permutations $\sigma_{k}$); see Section \ref{sec:open_questions}
for more details.
\end{rem}

\begin{rem}
In random $d$-regular graphs, balls of any constant radius $r$ around
an individual vertex are trees with high probability (even when $d$
is logarithmic in the number of vertices). Such sets are very poorly
connected -- the vertex expansion is of order $1/d^{r}$ (consider
cutting the $d$-ary tree in half at the central vertex). However,
in a random twisted hypercube, a ball of radius $r$ contains $G_{r}$
as a subgraph, which, by the theorem above, has good vertex expansion
with arbitrarily high probability for large $r$.
\end{rem}

\subsection{Eigenvalues}

Let $A\in\r^{m\times m}$ be a symmetric matrix, whose eigenvalues
are $\lambda_{1}\geq\lambda_{2}\geq\ldots\geq\lambda_{m}$. Let $\mu^{A}:=\frac{1}{m}\sum_{i=1}^{m}\delta_{\lambda_{i}}$
be the uniform measure over the eigenvalues of $A$, where $\delta_{s}$
is the Dirac-delta distribution centered at $s$. 

Let $\mathrm{Adj}\left(Q_{n}\right)$ be the adjacency matrix of the
hypercube $Q_{n}$. The $2^{n}$ eigenvalues and eigenvectors of $\mathrm{Adj}\left(Q_{n}\right)$
are well understood; the following is well known \cite[Section 1.4.6]{brouwer_haemers_spectra_of_graphs}.
\begin{fact}
\label{fact:eigenvalues_of_hypercube}For every integer $d\in\left[0,n\right]$,
the adjacency matrix $\mathrm{Adj}\left(Q_{n}\right)$ has eigenvalue
$n-2d$ with multiplicity ${n \choose d}$.
\end{fact}

In particular, the hypercube's largest eigenvalue is $n$, while its
second largest eigenvalue is $n-2$. Thus, its normalized spectral
gap, defined as $\frac{1}{n}\left(\lambda_{1}-\lambda_{2}\right)$,
is $\frac{2}{n}$. The same gap is achieved for the graphs $G_{n}$,
regardless of the choice of $\permvector$. The following proposition
is proved in Section \ref{subsec:postponed_eigenvalue_proofs}.
\begin{prop}
\label{prop:bad_gap_for_duplimatching}Let $A_{n}$ be the adjacency
matrix of $G_{n}$. Then $\lambda_{1}=n$ and $\lambda_{2}=n-2$. 
\end{prop}

A consequence of Fact \ref{fact:eigenvalues_of_hypercube} is that
$\mu^{\mathrm{Adj}\left(Q_{n}\right)}$ is the probability measure
of a $\left\{ \pm1\right\} $ Binomial random variable with $n$ trials
and success probability $1/2$. By the central limit theorem, we then
have that 
\[
\mu^{\mathrm{Adj}\left(Q_{n}\right)/\sqrt{n}}\to\Gamma
\]
weakly, where $\Gamma$ is the standard Gaussian distribution on $\r$.
Unlike the spectral gap, this property is not preserved for the random
twisted hypercube graph. In fact, the spectrum of $A=\mathrm{Adj}\left(\boldsymbol{G}_{n}\right)$
behaves like that of a random $n$-regular graph. 
\begin{thm}
\label{thm:eigenvalues_follow_semicircle_law}Let $\mu_{n}=\mu^{A/\sqrt{n}}$.
Then the random measure $\mu_{n}$ converges weakly to the semicircle
law $\circlaw$ in probability, i.e. the absolutely continuous measure
whose probability density function is 
\[
f_{\text{circ}}\left(x\right)=\begin{cases}
\frac{2}{4\pi^{2}}\sqrt{4-x^{2}} & x\in\left[-2,2\right],\\
0 & x\notin\left[-2,2\right].
\end{cases}
\]
\end{thm}

The above theorem follows from the following lemma, which states that
the number of short cycles in the neighborhood of any vertex in $\boldsymbol{G}_{n}$
is small. Essentially, this means that $\boldsymbol{G}_{n}$ is almost
locally treelike. For a vertex $v$ and positive integer $k$, let
$\theta\left(v,k\right)$ denote the number of cycles of length no
more than $k$ containing $v$, and $B\left(v,k\right)$ denote the
ball of radius $k$ around $v$.
\begin{lem}
\label{lem:small_number_of_cycles}Let $v\in V_{n}$ and let $k>0$
be an integer. There exists a constant $C>0$ which depends only on
$k$ such that the following holds. Let $m_{0}>0$ be an integer,
and let 

\[
F_{v}=\bigcup_{u\in B\left(v,k\right)}\left\{ \theta\left(u,k\right)\leq Cm_{0}^{k+1}\right\} .
\]
 Then
\[
\p\left[F_{v}\right]\geq1-C2^{-m_{0}}m_{0}^{2k+2}n^{2k+1}.
\]
\end{lem}

Theorem \ref{thm:eigenvalues_follow_semicircle_law} and Lemma \ref{lem:small_number_of_cycles}
are proven in Section \ref{subsec:postponed_eigenvalue_proofs}.
\begin{rem}
A classical theorem by McKay \cite{mckay_eigenvalues_for_large_regular_graphs}
states that a regular graph on $N$ vertices has a limiting semicircle
law if, for every $k$, the number of $k$-cycles in the graph is
$o\left(N\right)$. This result cannot be directly used in the case
of the twisted hypercube: for example, each vertex is guaranteed to
be in a $4$-cycle, so there are at least $N/4$ $4$-cycles in every
twisted hypercube (in fact, we conjecture that for the duplicube,
for every $k$, each vertex is in a constant number of $k$-cycles
in expectation). Lemma \ref{lem:small_number_of_cycles} is the main
technical component in our proof of Theorem \ref{thm:eigenvalues_follow_semicircle_law}.
\end{rem}

\subsection{Asymmetry of $\boldsymbol{G}_{n}$}

Let $G=\left(V,E\right)$ be any graph. A function $\varphi:V\to V$
is called an \emph{automorphism }of $G$ if $\left\{ x,y\right\} \in E\iff\left\{ \varphi\left(x\right),\varphi\left(y\right)\right\} \in E$.
The set of all automorphisms of a graph is denoted by $\aut\left(G\right)$,
and always contains the trivial automorphism -- the identity function
$\mathrm{Id}$. 

It is well known that for the hypercube, $\abs{\aut\left(Q_{n}\right)}=n!2^{n}$,
and every automorphism $\varphi\left(x\right)$ is of the form $\varphi\left(x_{1},\ldots,x_{n}\right)=\left(x_{\pi\left(1\right)}+b_{1},\ldots,x_{\pi\left(n\right)}+b_{n}\right)$
for some permutation $\pi\in S_{n}$ and $b\in\left\{ 0,1\right\} ^{n}$.
On the other hand, a random regular graph of degree $n$ on $2^{n}$
vertices is almost surely asymmetric, i.e. almost surely has no non-trivial
automorphisms \cite[Corollary 3.5]{mckay_wormald_automorphisms}.
This is also true for random twisted hypercubes.
\begin{thm}
\label{thm:asymmetry_of_g_n}
\[
\p\left[\aut\left(\boldsymbol{G}_{n}\right)=\left\{ \mathrm{Id}\right\} \right]=1-O(n^{2}2^{-n/20}).
\]
\end{thm}

The proof of Theorem \ref{thm:asymmetry_of_g_n} is found in Section
\ref{subsec:postponed_asymmetry_proofs}.

\subsection{Different base graphs}

The twisted hypercube graph is the result of repeatedly applying the
$\sigma$-twist operation on a single vertex. It is also possible
to start with any base graph $G_{0}=H$, and define $G_{n}^{H}=\twist{G_{n-1}^{0}}{G_{n-1}^{1}}{\sigma_{n-1}}$,
where $\sigma_{n-1}$ is a permutation on $2^{n-1}\abs{V\left(H\right)}$
vertices. When each $\sigma_{k}$ is a uniformly random permutation
on $2^{k}\abs{V\left(H\right)}$ elements, we denote the resulting
random graph by $\boldsymbol{G}_{n}^{H}$. In this case we say that
$\boldsymbol{G}_{n}^{H}$ is a random twisted hypercube with base
graph $H$. None of the main results concerning the diameter, expansion,
and eigenvalues are severely affected. This is because as $n\to\infty$,
the vast majority of the edges meeting each vertex are those created
by the $\sigma$-twist operation.
\begin{lem}
Let $H$ be a finite connected graph. Let $\boldsymbol{G}_{n}^{H}$
be random twisted hypercube with base graph $H$. Then there exists
a random twisted hypercube $\boldsymbol{G}_{n}$ and a coupling $(\boldsymbol{G}_{n}^{H},\boldsymbol{G}_{n})$
such that:
\end{lem}

\begin{enumerate}
\item $D\left(\boldsymbol{G}_{n}^{H}\right)\leq D\left(H\right)D\left(\boldsymbol{G}_{n}\right).$
\item If the permutations that define $\boldsymbol{G}_{n}^{H}$ are independent
then so are the permutations that define $\boldsymbol{G}_{n}$.
\item If $\boldsymbol{G}_{n}^{H}$ is the duplicube with base graph $H$,
then $\boldsymbol{G}_{n}$ is also a duplicube.
\end{enumerate}
\begin{proof}[Proof sketch]
Consider the $\sigma$-twist operation $\boldsymbol{G}_{k+1}^{H}=\twist{\boldsymbol{G}_{k}^{H,0}}{\boldsymbol{G}_{k}^{H,1}}{\sigma_{k}}$.
By contracting each copy of $H$ in $\boldsymbol{G}_{k+1}^{H}$ to
a single vertex and using Hall's marriage theorem, there exists a
set $S$ of $2^{k}$ edges induced by $\sigma_{k}$ which comprise
a perfect matching between the copies of $H$ in the two instances
of $\boldsymbol{G}_{k}^{H}$. Such a set $S$ naturally induces a
permutation on $\left\{ 0,1\right\} ^{k}$. For a given $\sigma_{k}$,
let $\pi_{k}$ be chosen uniformly at random among all such induced
permutations. Then if $\sigma_{k}$ is chosen uniformly then $\pi_{k}$
is a uniform random permutation on $\left\{ 0,1\right\} ^{k}$. We
can use these permutations to generate a graph $\boldsymbol{G}_{n}$
which is coupled with $\boldsymbol{G}_{n}^{H}$ so that $\boldsymbol{G}_{n}$
is a subgraph of the graph obtained by contracting every copy of $H$
in $\boldsymbol{G}_{n}^{H}$ to a single vertex. Thus $D\left(\boldsymbol{G}_{n}^{H}\right)\leq D\left(H\right)D\left(\boldsymbol{G}_{n}\right)$.
\end{proof}
Thus, both Theorem \ref{thm:small_diameter_duplicube} and Theorem
\ref{thm:small_diameter_general_case} continue to hold with only
a constant-factor change in the diameter. 

\begin{rem}
If $H$ is not connected, one may simply apply the $\sigma$-twist
operation several times first until $\boldsymbol{G}_{n}^{H}$ is connected
(this can be shown to happen with probability tending to $1$ as $n\to\infty$),
then use that as the base graph. 
\end{rem}

\begin{cor}[Corollary to Theorem \ref{thm:good_vertex_expander}]
Let $H$ be a finite graph. For every $\eta\in(0,1)$ there exists
a constant $\alpha>0$ such that the 
\[
\lim_{n\to\infty}\p[\boldsymbol{G}_{n}\text{is a }(\eta,\alpha)\text{-vertex expander}]=1.
\]
\end{cor}

The proof of the above corollary is essentially identical to that
of Theorem \ref{thm:good_vertex_expander}. The latter only uses the
edges created by the last three $\sigma$-twist operations, and so
the statement still holds for $\boldsymbol{G}_{n}^{H}$ as well.

\begin{cor}[Corollary to Theorem \ref{thm:eigenvalues_follow_semicircle_law}]
Let $H$ be a finite graph. Let $\mu_{n}=\mu^{A\left(\boldsymbol{G}_{n}^{H}\right)/\sqrt{n}}$.
Then $\mu_{n}$ converges weakly to the semicircle law $\circlaw$
in probability, i.e. the absolutely continuous measure whose probability
density function is 
\[
f_{\text{circ}}\left(x\right)=\begin{cases}
\frac{2}{4\pi^{2}}\sqrt{4-x^{2}} & x\in\left[-2,2\right]\\
0 & x\notin\left[-2,2\right].
\end{cases}
\]
\end{cor}

\begin{proof}[Proof sketch]
We will assume for simplicity that $\abs{V\left(H\right)}=2^{d}$
for some integer $d$. We can couple $\boldsymbol{G}_{n}^{H}$ with
$\boldsymbol{G}_{n+d}$ by observing that $\boldsymbol{G}_{n+d}=\boldsymbol{G}_{n}^{\boldsymbol{G}_{d}}$,
and using the same permutations $\sigma_{k}$ for $\boldsymbol{G}_{n}^{H}$
and $\boldsymbol{G}_{n}^{\boldsymbol{G}_{d}}$. Since all the edges
due to the permutations are the same for $\boldsymbol{G}_{n}^{H}$
and $\boldsymbol{G}_{n}^{\boldsymbol{G}_{d}}$, their adjacency matrices
differ by no more than $c:=\abs{V\left(H\right)}$ entries at each
row, and all the eigenvalues of the matrix $\Delta=\mathrm{Adj}\left(\boldsymbol{G}_{n}^{H}\right)-\mathrm{Adj}\left(\boldsymbol{G}_{n}^{\boldsymbol{G}_{d}}\right)$
are bounded by $c$. Denoting $A:=\mathrm{Adj}\left(\boldsymbol{G}_{n}^{\boldsymbol{G}_{d}}\right)$
, for every integer $k>0$ we have
\begin{align*}
\abs{\sum_{i=1}^{2^{n+d}}\lambda_{i}\left(\boldsymbol{G}_{n}^{\boldsymbol{G}_{d}}\right)^{k}-\sum_{i=1}^{2^{n+d}}\lambda_{i}\left(\boldsymbol{G}_{n}^{H}\right)^{k}} & =\abs{\tr\left(A^{k}\right)-\tr\left(\left(A+\Delta\right)^{k}\right)}\\
 & =\abs{\tr\left(P\left(A,\Delta\right)\right)},
\end{align*}
where $P$ is a polynomial of degree $k$ for which in every monomial,
$A$ has total degree at most $k-1$. By Von Neumann's trace inequality
\cite[eq H.10]{marshall_inequalities}, if $A_{1},\ldots,A_{m}$ are
$N\times N$ symmetric matrices, then 
\[
\sum_{i=1}^{N}\lambda_{i}\left(A_{1}\cdots A_{m}\right)\leq\sum_{i=1}^{N}\lambda_{i}\left(A_{1}\right)\cdots\lambda_{i}\left(A_{m}\right),
\]
and so the trace of every monomial in $P$ is bounded above by $c^{k}\sum_{i=1}^{2^{n+d}}\abs{\lambda_{i}\left(A\right)^{k-1}}$.
Thus the difference in the normalized moments of $\boldsymbol{G}_{n}^{\boldsymbol{G}_{d}}$
and $\boldsymbol{G}_{n}^{H}$ is bounded by
\begin{align*}
\frac{\left(n+d\right)^{-k/2}}{2^{n+d}}\abs{\sum_{i=1}^{2^{n+d}}\lambda_{i}\left(\boldsymbol{G}_{n}^{\boldsymbol{G}_{d}}\right)^{k}-\sum_{i=1}^{2^{n+d}}\lambda_{i}\left(\boldsymbol{G}_{n}^{H}\right)^{k}} & \leq C\left(k\right)\frac{\left(n+d\right)^{-k/2}}{2^{n+d}}\sum_{i=1}^{2^{n+d}}\abs{\lambda_{i}\left(A\right)^{k-1}}\\
\left(\text{Cauchy-Schwarz}\right) & \leq C\left(k\right)\sqrt{\frac{1}{2^{n+d}}n^{-1}\sum_{i}\lambda_{i}\left(A\right)^{2}}\sqrt{\frac{1}{2^{n+d}}n^{-\left(k-1\right)}\sum_{i}\lambda_{i}\left(A\right)^{2k-4}}.
\end{align*}
By the proof of Theorem \ref{thm:eigenvalues_follow_semicircle_law},
$\frac{1}{2^{n+d}}n^{-\left(k-2\right)}\sum_{i}\lambda_{i}\left(A\right)^{2k-4}$
converges to a constant in probability as $n\to\infty$, which means
that the sum on the right-hand side above converges to $0$ in probability.
This implies that the $k$-th moments of the empirical distribution
of the eigenvalues of $\boldsymbol{G}_{n}^{H}$ converge to those
of the semicircle law.
\end{proof}
\begin{cor}[Corollary to Theorem \ref{thm:asymmetry_of_g_n}]
\label{cor:asymmetry-different-base-graph}Let $H$ be a finite graph.
Then 
\[
\lim_{n\to\infty}\p\left[\aut\left(\boldsymbol{G}_{n}^{H}\right)=\left\{ \mathrm{Id}\right\} \right]=1-O\left(n^{2}2^{-\frac{n}{20}}\right).
\]
\end{cor}

We omit the proof of Corollary \ref{cor:asymmetry-different-base-graph}
since it is similar to the proof of Theorem \ref{thm:asymmetry_of_g_n}.

\section{Remarks and further directions\label{sec:open_questions}}
\begin{enumerate}
\item In \cite{zhu_hypercube_variant_with_small_diameter}, Zhu gives a
simple-to-define, deterministic sequence of permutations $\bar{\sigma}=\left(\sigma_{k}\right)_{k=1}^{\infty}$
for which the twisted hypercube has asymptotically optimal diameter.
What can be said about the expansion, asymmetry, and eigenvalues of
this construction? If these properties differ from those of a random
twisted hypercube, find a deterministic construction for which they
agree.
\item Theorem \ref{thm:small_diameter_duplicube} shows that the random
duplicube has diameter $\frac{n}{\log_{2}n}(1+o(1)).$ Is it true
that for all random twisted cubes the same result holds with high
probability?
\item Given a sequence of permutations $\bar{\sigma}=\left(\sigma_{k}\right)_{k=1}^{\infty}$,
is there a good local routing scheme for the duplicube $G_{n}\left(\bar{\sigma}\right)$
that gives a better approximation than Proposition \ref{prop:trivial_upper_bound}
to the shortest path between two vertices?
\item The twisted-hypercube model can be readily extended to $d$-dimensional
hypergraphs: at every step, create $d$ instances of the current hypergraph,
and connect the vertices of the $d$ instances by a perfect matching
of $d$-hyperedges. What can be said about the resultant hypergraph? 
\item The graph $G_{n}$ is, in general, not a good edge-expander. One reason
for this are cuts across the matchings $\sigma_{k}$ for large $k$.
For example, the two instances of $G_{n-1}$ in $G_{n}$ each have
$2^{n-1}$ vertices, and are connected by $2^{n-1}$ edges, giving
an isoperimetric ratio of $1$. This is not so large for a graph whose
degree is $n$. What can we say about the geometric properties of
a set with small edge boundary? For $Q_{n}$ it is known that sets
that have small edge expansion are similar to subcubes \cite{ellis2011almost}.
Do non-expanding sets in $\boldsymbol{G}_{n}$ have similar structure?
\item We show that with high probability $\boldsymbol{G}_{n}$ is a good
vertex expander. However, to our knowledge there is no efficient way
to verify that a given graph is a vertex expander: assuming the Small-Set-Expansion
Hypothesis, it is hard to even approximate the vertex expansion of
a graph in polynomial time \cite{louis_raghavendra_vempala_vertex_expansion_hardness}.
Is it possible to exploit the structure of the twisted hypercube to
verify this property in time $\mathrm{poly}\left(2^{n}\right)$?
\item By using the same coupon-collector argument as for the hypercube,
the mixing time of the lazy simple random walk of any twisted hypercube
is $O\left(n\log n\right)$. On the other hand, if an $n$-th generation
edge is never refreshed, then the random walk stays constrained to
one half of the graph, and so the mixing time must also be $\omega\left(n\right)$.
What is the mixing time for the lazy simple random walk on $\boldsymbol{G}_{n}$?
Is it $o\left(n\log n\right)$ with high probability?
\item Is it possible to remove edges from $\boldsymbol{G}_{n}$ and obtain
a (near) constant-degree graph, while maintaining good vertex expansion?
Is it possible to approach the vertex expansion of a constant-degree
random regular graph in this way?
\item Replace every vertex of $Q_{n}$ by an $n$-cycle, obtaining a graph
$CCC_{n}$; this is known as the cube-connected-cycle \cite{preparata_vuillemin_ccc}.
As $n\to\infty$, it is well known that $CCC_{n}$ converges in the
Benjamini-Schramm sense \cite{benjamini_schramm_convergence} to the
lamplighter graph $\z_{2}\wr\z$. We conjecture that the Benjamini-Schramm
limit of the twisted cube-connected-cycle, obtained by replacing every
vertex of $\boldsymbol{G}_{n}$ by an $n$-cycle, is the $3$-regular
tree: as $n\to\infty$, a vertex chosen at random from this graph
corresponds to a high-generation edge with high probability, and these
should not be part of many small cycles. 
\item Although Theorem \ref{thm:asymmetry_of_g_n} shows that random permutations
lead to an asymmetric graph, in general different choices of $\bar{\sigma}$
can lead to different automorphism groups. Can we relate properties
of the automorphism group of the duplicube $G_{n}\left(\bar{\sigma}\right)$
with properties of $\bar{\sigma}$? In particular, can we find large
families of $\bar{\sigma}$ so that $G_{n}\left(\bar{\sigma}\right)$
is vertex-transitive? As a non-trivial example, consider the permutations
$\sigma_{k}=\mathrm{Id}$ for $k\neq2$. There are two essentially
different possibilities for $\sigma_{2}$: the first is $\sigma_{2}=\mathrm{Id}$,
leading to the hypercube $Q_{n}$. The second is the matching between
a pair of $4$-cycles which sends an edge to a non-edge. This leads
to a vertex-transitive graph that is not isomorphic to $Q_{n}$. Can
we find a (perhaps random) vertex-transitive $G_{n}\left(\bar{\sigma}\right)$
with improved geometric properties over the hypercube? 
\item The argument in Theorem \ref{thm:good_vertex_expander} only uses
the edges of the last three generations of the twisted hypercube.
On the other hand, such an argument could not hold while using only
the edges of the last two generations, since the graph induced by
the edges of the last two generations is a union of cycles. In fact,
we believe that when $\sigma_{i}=\mathrm{Id}$ for $i<n-2$, the resultant
graph does not have constant vertex-expansion with high probability.
In light of this, it is natural to ask: for an integer $k>0$, what
are the properties of the twisted hypercube graph, where $\sigma_{i}=\mathrm{Id}$
for $i<n-k$, and $\sigma_{i}$ is uniformly random for $i\geq n-k$?
What happens when $k$ grows slowly to infinity with $n$? This is
a natural interpolation between the hypercube $Q_{n}$ and the completely
random twisted hypercube $\boldsymbol{G}_{n}$.
\item The hypercube $Q_{n}$ induces a partial order on its vertices in
a natural way: $x\leq y$ if $x_{i}\leq y_{i}$ for every $i$. This
natural partial order has applications (see e.g. \cite[Chapter 6]{van_lint_wilson_course_in_combinatorics}).
The twisted hypercube induces a similar partial order inductively:
given the order on $G_{n-1}$, extend it to $G_{n}$ by having $\left(x,0\right)<\left(\sigma_{n-1}\left(x\right),1\right)$
for all $x\in V_{n-1}$, and by keeping the original order within
$G_{n-1}$ in both instances. It can be verified that this is indeed
a partial order. What are the properties of this partial order as
a function of $\bar{\sigma}$? Are there any combinatorial applications
to the partial order produced by the twisted hypercube?
\item The hypercube $Q_{n}$ is bipartite, and hence always 2-colorable.
On the other hand, the chromatic number $\chi$ of random $d$-regular
graphs (of constant degree $d$) is known to take only one of two
possible values with high probability, and satisfies $2\chi\log\chi\approx d$
\cite{kemkes_perez_wormald_chromatic_number}. What is the chromatic
number of $\boldsymbol{G}_{n}$? It can be shown that it is at least
$3$ with high probability, and so $\boldsymbol{G}_{n}$ is in general
not bipartite (to see this, consider the case where $\boldsymbol{G}_{n}$
is bipartite, and look at the probability that $\sigma_{n}$ induces
an odd cycle in $\boldsymbol{G}_{n+1}$).
\end{enumerate}

\section{Proofs\label{sec:postponed_proofs}}

\subsection{Notation and definitions}

All logarithms are in base $e$ unless otherwise noted. For two sequences
$f\left(n\right)$, $g\left(n\right)$, we write $f=o\left(g\right)$
if $\lim\abs{f\left(n\right)}/\abs{g\left(n\right)}\to0$. For natural
numbers $n,k\in\n$, we set $N=2^{n}$ and $K=2^{k}$. The set of
numbers $1,\ldots,n$ is denoted by $\left[n\right]$. 

We denote the vertex set of $G_{n}$ by $V_{n}:=\left\{ 0,1\right\} ^{n}$.
For two sets $S_{1},S_{2}\subseteq V_{n}$, write $S_{1}\sim S_{2}$
if there are $x\in S_{1}$ and $y\in S_{2}$ with $\left\{ x,y\right\} \in E\left(G_{n}\right)$,
and $S_{1}\not\sim S_{2}$ otherwise. We say that the edges between
two disjoint sets of vertices $A,B\subseteq V_{n}$ constitute a \emph{matching}
if every vertex in $A\union B$ is adjacent to at most one such edge.

Let $x,y\in V_{n}$. The \emph{generation number} of $x$ and $y$,
denoted by $\gamma\left(x,y\right)$, is defined as 
\[
\gamma\left(x,y\right):=n-\max\left\{ 1\leq s\leq n\mid x_{i}=y_{i}\all i\geq s\right\} ,
\]
i.e. $n$ minus the longest common suffix of $x$ and $y$. If $\left\{ x,y\right\} \in E\left(G_{n}\right)$
is an edge, then that edge is due to the permutation $\sigma_{\gamma\left(x,y\right)-1}$.
Supposing that $\gamma\left(x,y\right)=k$, we then say that $x$
and $y$ are $k$-neighbors. Every vertex $x$ has exactly one $k$-neighbor
for every $k\in\left[n\right]$; we denote it by $N_{k}\left(x\right)$.

For an integer $r>0$ and vertex $v\in V_{n}$, denote by
\[
B\left(v,r\right):=\left\{ z\in V_{n}\mid\exists\text{ a path of at most \ensuremath{r} edges from \ensuremath{v} to \ensuremath{z}}\right\} 
\]
the ball of radius $r$ around $v$, and by
\[
B_{<k}\left(v,r\right):=\left\{ z\in V_{n}\mid\exists\text{ a path \ensuremath{P} of at most \ensuremath{r} edges from \ensuremath{v} to \ensuremath{z} s.t. \ensuremath{\gamma\left(x,y\right)<k} \ensuremath{\all\left\{  x,y\right\} } \ensuremath{\in E\left(P\right)} }\right\} 
\]
the $r$-neighborhood of $v$ obtained by paths which only use edges
of generations smaller than $k$.

For $1\leq s<n$, the graph $G_{n}$ contains multiple disjoint instances
of graphs $G_{s}$. Indeed, let $z\in\left\{ 0,1\right\} ^{n-s}$,
and define 
\begin{equation}
V_{n}^{z}:=\left\{ \left(y,z\right)\in V_{n}\Bigl|y\in\left\{ 0,1\right\} ^{s}\right\} .\label{eq:disjoint_copy_definition}
\end{equation}
Then the induced graph on $V_{n}^{z}$ is an instance of $G_{s}$
(when the construction is deterministic, or in the case of the duplicube,
these instances are all isomorphic). The sets $V_{n}^{z}$ are disjoint
for different $z$, and partition the vertices of $G_{n}$. For a
vertex $x\in V_{n}$, let $I_{s}\left(x\right)$ be the set $V_{n}^{z}$
which contains $x$; it is the set of all vertices in $G_{n}$ which
share a suffix with $x$ of size at least $n-s$, i.e.,

\[
I_{s}\left(x\right):=\left\{ y\in V_{n}\mid\gamma\left(x,y\right)\leq s\right\} .
\]
Note that $\abs{I_{s}\left(x\right)}=2^{s}$. See Figure \ref{fig:containing_subgraphs}
for a visual aid. Finally, for a set $S\subseteq V_{n}$, we denote
by $\partial_{k}S$ the boundary due to the first $k$ generations
of edges, i.e.
\[
\partial_{k}S=\left\{ x\notin S\mid\exists y\in S,\left\{ x,y\right\} \in E\left(G_{n}\right),\gamma\left(x,y\right)\leq k\right\} .
\]
We often write $\partial S$ instead of $\partial_{n}S$ for brevity.
\begin{figure}[H]
\begin{centering}
\includegraphics[scale=0.33]{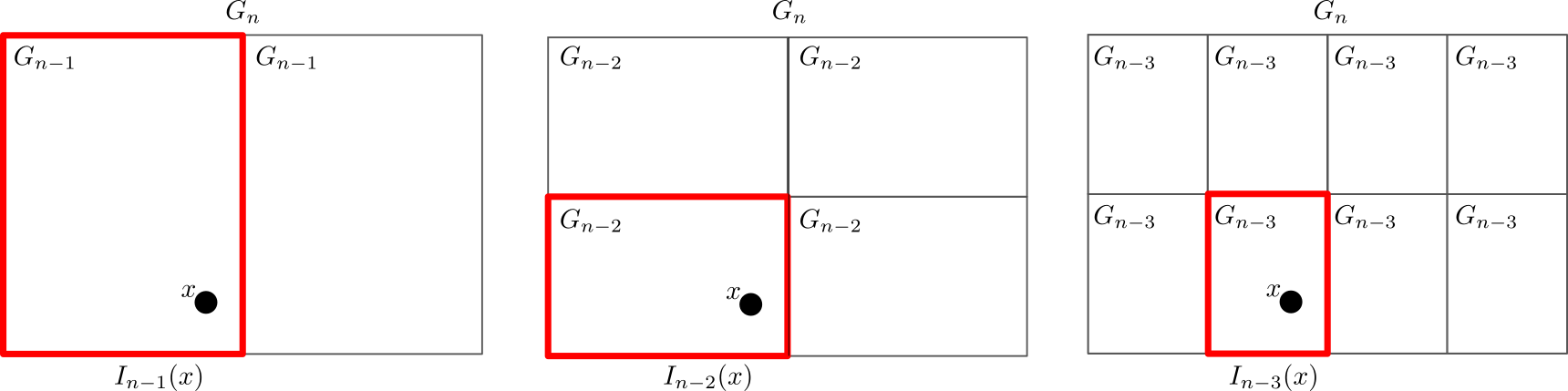}
\par\end{centering}
\caption{\label{fig:containing_subgraphs}Each large rectangle represents the
same graph $G_{n}$, with the same vertex $x$ highlighted. The partition
into instances of $G_{s}$ is shown for $s\in\left\{ n-1,n-2,n-3\right\} $
while highlighting $I_{s}\left(x\right)$.}
\end{figure}

\subsection{The diameter\label{subsec:postponed_diameter_proofs}}

The proof of Theorem \ref{thm:small_diameter_duplicube} resembles
the proof of Dudek et al. \cite{dudek_et_al_randomly_twisted_hypercube}
for the independent twisted hypercube.
\begin{proof}[Proof of Theorem \ref{thm:small_diameter_duplicube}]

The main idea of the proof is to show that with high probability,
for every $v\in V_{n}$, the ball around $v$ of radius $\frac{n}{2\log_{2}n}+O\left(\frac{n}{\log_{2}^{2}n}\right)$
contains $\geq n2^{n/2}$ vertices in the copy of $\boldsymbol{G}_{n-1}$
which contains $v$. If this holds, then the diameter is $\frac{n}{\log_{2}n}+O\left(\frac{n}{\log_{2}^{2}n}\right)$:
for every $v\in V_{n}^{0},u\in V_{n}^{1}$, denote by $S_{v}^{n-1},S_{u}^{n-1}$
the balls around $v,u$ in $\boldsymbol{G}_{n-1}$. The probability
that $S_{v}^{n-1},S_{u}^{n-1}$ are connected by the last permutation
is 
\[
1-\frac{{2^{n-1}-n2^{n/2} \choose n2^{n/2}}}{{2^{n-1} \choose n2^{n/2}}}\geq1-\left(1-\frac{n2^{(n-1)/2}}{2^{n-1}}\right)^{n2^{\left(n-1\right)/2}}\geq1-e^{-n^{2}}.
\]
By a union bound, with high probability every two such balls are connected,
so we can find a path of length $\frac{n}{\log_{2}n}+O\left(\frac{n}{\log_{2}^{2}n}\right)$
from $v$ to $u$. If $v,u$ are are in the same $V_{n}^{i}$ then
we use the path from $v$ to $\sigma_{n-1}\left(u\right)$ and go
from $\sigma_{n-1}\left(u\right)$ to $u$ via one additional edge.

Hence we shall show that for any fixed $v\in V$, the ball of radius
$\frac{n}{2\log_{2}n}+O\left(\frac{n}{\log_{2}^{2}n}\right)$ contains
$\geq n2^{n/2}$ vertices in the copy of $\boldsymbol{G}_{n-1}$ which
contains $v$, with probability $1-o\left(n2^{-2n}\right)$. Beginning
with an empty graph on $V=\left\{ 0,1\right\} ^{n}$, we add edges
by revealing the values of the permutations $\sigma_{k}$ one-by-one,
as follows. Set $q=0.9n$. 
\begin{enumerate}
\item Initiate a queue $Q$ and insert $v\in Q$.
\item While $Q\ne\emptyset$:\\
Take out the first $u\in Q$ and for every $i=q,q+1,...,n-2$: 
\begin{enumerate}
\item \label{enu:revealing_the_edges}Let $u=(u_{1},b,u_{2})$ so that $u_{1}\in\{0,1\}^{i}$.
If $b=0$, set $\pi=\sigma_{i}$, and otherwise set $\pi=\sigma_{i}^{-1}$;
then, if $\pi\left(u_{1}\right)$ wasn't previously revealed, reveal
it. This is called the \emph{revealing step}.
\item For every $u'\in\{0,1\}^{n-1-i}$, we add all edges $\left\{ (u_{1},b,u'),(\pi\left(u_{1}\right),1-b,u')\right\} $.
\item For every vertex $w$ that was connected to $u$ and was not previously
added to $Q$, we add $w\in Q$.
\end{enumerate}
\item After $Q$ is empty reveal all other edges in an arbitrary order.
\end{enumerate}
We note that when revealing an entry $\sigma_{i}\left(u\right)$ or
$\sigma_{i}^{-1}\left(u\right)$, we in fact add $2^{n-1-i}$ edges
to $\boldsymbol{G}_{n}$ that come from the different copies of the
$i$-th generation duplicube. We say that a vertex $u$ is \emph{discovered}
at step $k$, if one of the edges revealed in the $k$-th step is
the first edge that is adjacent to $u$ (where $k$ refers to the
number of times we have done step (\ref{enu:revealing_the_edges})).
Let $G'$ be the subgraph of $\boldsymbol{G}_{n}$ whose edges are
only the edges of generations $i\geq q$. Let $S_{j}$ be the set
of vertices $u\in V$ so that $d_{G'}\left(v,u\right)=j$ and so that
if $u$ was discovered at step $k$, then no vertex with the same
$0.9n$-prefix of $u$ was discovered previously. Finally, fix $r_{0}$
to be the smallest integer so that $n^{r_{0}}\geq2^{0.1n}$ and let
$r_{1}$ be the smallest integer so that $\left(n/1000\right)^{r_{1}}\geq1000n2^{n/2}$.
Clearly $r_{0}=\frac{0.1n}{\log_{2}n}+O\left(1\right)$ and $r_{1}=\frac{n}{2\log_{2}n}+O\left(\frac{n}{\log_{2}^{2}n}\right)$. 

We will analyze the growth of $F(j)=\abs{S_{j}}$ separately for $j\leq r_{0}$
and $r_{0}<j\leq r_{1}$ starting with $j\leq r_{0}$. We say that
the $k$-th step is \emph{bad} if $u=(u_{1},b,u_{2})$ was the vertex
taken out of $Q$, the value $x=\sigma_{i}^{\pm1}(u_{1})$ was revealed,
and there exists a vertex $w$ whose prefix is $x$ that was discovered
in a previous step. At the first phase, we will show that there are
very few bad steps. First we calculate a bound on the number of steps
$m$ while the distance between $v$ and the vertex that was taken
out of the queue in the step is of distance $j\leq r_{0}$. Namely,
\[
m\leq\left(n-q\right)\sum_{j=1}^{r_{0}}\left(n-q-1\right)^{j}\leq2n\left(0.1n\right)^{r_{0}}\leq2^{0.11n}.
\]
Moreover, for the $\ell$-th step for $\ell\leq m$, the probability
that $\ell$ is a bad step is at most
\[
p_{\ell}\leq\frac{\ell\left(n-q\right)+1}{2^{0.9n}-1-\ell\left(n-q\right)},
\]
where the numerator is the number of previously discovered vertices
(which upper bounds the number of prefixes  discovered), and the denominator
is the number of choices left. We note that the probability of choosing
a given prefix is not uniform, but if a given prefix has already been
chosen it only decreases its probability to be chosen again. This
is at most
\[
p_{\ell}\leq\frac{n2^{0.11n}}{2^{0.899n}}\leq2^{-0.78n}.
\]
As this bound is uniform for all $\ell$ and the same bound holds
true for the conditional probability subject to any way of revealing
the first $(\ell-1)$ edges, the probability $p$ of having $c$ bad
steps in the first phase is at most
\[
p\leq\binom{m}{c}2^{-0.78nc}\leq2^{0.11nc-0.78nc}=2^{-0.67cn}.
\]
Taking $c=4$ we get that this probability is $o\left(2^{-2n}\right).$
Let $k$ be a step where the vertex taken out of $Q$ is in $S_{j-1}$.
If this $k$-th step is not bad, then $S_{j}$ grows by $1$ due to
a new vertex discovered, and if there is a bad edge then it reduces
the size by at most 2 (since there were at most two prefixes involved
in choosing the bad step). Therefore, 

\[
F\left(1\right)\geq n-q-2c\geq\left(n/1000\right),
\]
\[
F\left(2\right)\geq\left(F\left(1\right)-2c\right)\left(n-q\right)\geq0.1n\left(0.001n-8\right)\geq\left(n/1000\right)^{2},
\]

\[
...
\]

\[
F\left(r_{0}\right)\geq\left(F\left(r_{0}-1\right)-2c\right)\left(n-q\right)\geq\left(\left(n/1000\right)^{r_{0}-1}-8\right)0.1n\geq\left(n/1000\right)^{r_{0}}.
\]
During the second phase, we don't expect there to be no bad steps,
but as the set $S_{j}$ is already quite large, we expect that $S_{j}$
will still grow by an $\frac{n}{1000}$-factor. Indeed, conditioned
on $F\left(j\right)\geq\left(n/1000\right)^{j}\geq2^{0.005n},$ we
show that $F\left(j+1\right)\geq\frac{n}{1000}F\left(j\right)$ with
probability $\geq1-o\left(2^{-2n}\right)$. When all these events
occur, we can conclude that $F\left(r_{1}\right)\geq n2^{n/2}$ with
probability $1-o\left(n2^{-2n}\right).$ Fix $j>r_{0}$ and let $X$
be a random variable counting the number of bad steps exposed from
the vertices of $S_{j}$. The number of new vertices we discovered
up to this step is at most $n^{j}$ (and this is also a bound for
the number of prefixes discovered). For every step in this phase,
the probability that it is bad is at most
\[
p'\leq\frac{n^{j}}{2^{0.9n}-n^{j}},
\]
and as $n^{j}\leq n^{r_{1}}\leq2^{0.51n}$, this is at most $\frac{n^{j}}{2^{0.6n}}$.
We can bound $X$ from above with a $\left(F\left(j\right),\frac{n^{j}}{2^{0.6n}}\right)$-binomially
distributed random variable. Thus $\e[X]\leq F(j)\frac{n^{j}}{2^{0.6n}}=o(F(j))$.
Furthermore, by Chernoff's bound on binomial variables 
\[
\p\left[X\geq F\left(j\right)/10000\right]\leq e^{-\Omega\left(F\left(j\right)\right)}\leq o\left(2^{-2n}\right).
\]
When this event doesn't occur, then
\[
F\left(j+1\right)\geq\left(F\left(j\right)-2X\right)\left(n-q\right)\geq0.9998\cdot F\left(j\right)\cdot0.1n\geq F\left(j\right)\frac{n}{1000}
\]
as required.
\end{proof}
\begin{rem*}
As mentioned before, the proof of Theorem \ref{thm:small_diameter_duplicube}
is an adaptation of the proof in Dudek et al. \cite{dudek_et_al_randomly_twisted_hypercube}
for the small diameter of the independent twisted hypercube. The main
change is that in this proof, revealing the value of a permutation
$\sigma_{k}\left(u_{1}\right)$ (for $u_{1}\in\{0,1\}^{k}$) reveals
many edges between $\left(u_{1},0,u_{2}\right)$ to $\left(\sigma_{k}\left(u_{1}\right),1,u_{2}\right)$.
Hence, instead of accounting for the number of vertices already discovered,
we account for the number of prefixes already discovered.
\end{rem*}
\begin{proof}[Proof of Theorem \ref{thm:small_diameter_general_case}]
The proof is by induction. For the base cases, let $\gamma>0$ to
be chosen later, and let $k^{*}=\gamma\log n\log\log n$. by Proposition
\ref{prop:trivial_upper_bound}, for all $n\leq k^{*}$ we have 
\begin{align*}
D\left(\boldsymbol{G}_{n}\right) & \leq n\leq\log\log k^{*}\frac{n}{\log\log n}\\
 & \leq C\log\log\log n\frac{n}{\log\log n}
\end{align*}
and so (\ref{eq:n_over_log_n_diameter}) holds with probability $1$
for $C$ large enough (depending on $\gamma$).

For the induction step, let $n>k^{*}$. By increasing $\gamma$, we
may assume that $n$ is larger than any given global constant; this
will ensure that inequalities which hold only when $n$ is large enough
indeed hold. For an integer $k\geq1$ and $z\in\left\{ 0,1\right\} ^{n}$,
let $\boldsymbol{G}_{k}^{z}$ be the induced graph on $I_{k}\left(z\right)$;
this is an instance of $\boldsymbol{G}_{k}$. Denote by $E_{k}^{z}$
the event that $D\left(\boldsymbol{G}_{k}^{z}\right)\leq Ck\frac{\log\log\log k}{\log\log k}$,
and assume that $E_{k}^{z}$ holds for every $k=1,\ldots,n-1$ and
every $z\in\left\{ 0,1\right\} ^{n}$. Let $x,y\in V_{n}$. If $x_{n}=y_{n}$,
i.e. the two vertices are in the same half of the graph $\boldsymbol{G}_{n}$,
then by the induction hypothesis, $D\left(\boldsymbol{G}_{n-1}^{x}\right)\leq C\left(n-1\right)\frac{\log\log\log\left(n-1\right)}{\log\log\left(n-1\right)}$,
and we certainly have $d_{\boldsymbol{G}_{n}}\left(x,y\right)\leq Cn\frac{\log\log\log n}{\log\log n}$.

For the case $x_{n}=1-y_{n}$, i.e. the two vertices are in opposite
sides of the graph $\boldsymbol{G}_{n}$, we'll show that for a not-too-large
radius, the spheres around $x$ and $y$ contain enough vertices,
so that with high probability there is an edge between them induced
by $\sigma_{n-1}$. 

Given a vertex $v\in V_{n}$ and any integers $t\geq s\geq0$, let
$M\left(v,s,t\right)=\left\{ \left(N_{k}\left(v\right),k\right)\mid s\leq k\leq t\right\} $
be the set of neighbors of $v$ whose edge to $v$ was added at times
$s\leq k\leq t$, along with their generation number. Note that for
$\left(z,k\right)\in M\left(v,s,t\right)$, the set $I_{k}\left(z\right)$
is contained in $I_{t+1}\left(v\right)$, and that since each $k$-neighbor
is added at a different generation, the sets $\left\{ I_{k}\left(z\right)\right\} _{\left(z,k\right)\in M\left(v,s,t\right)}$
are all mutually disjoint (see Figure \ref{fig:branching_out}). We
can therefore iteratively apply the function $M\left(v,s,t\right)$
to obtain a large set of disjoint vertices. 

\begin{figure}[H]
\begin{centering}
\includegraphics[scale=0.33]{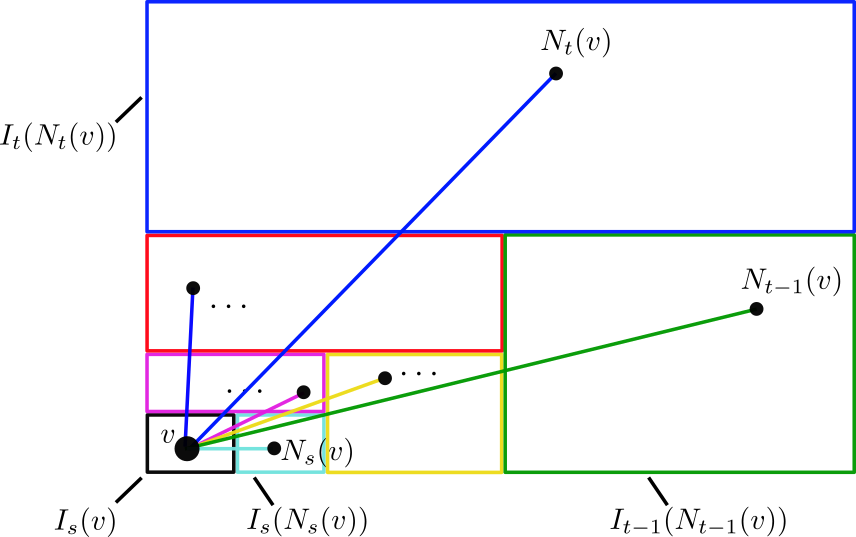}
\par\end{centering}
\caption{\label{fig:branching_out}The entire rectangle represents the graph
$I_{t+1}\left(v\right)$. Each neighbor $N_{k}\left(v\right)$ is
contained in $I_{k}\left(N_{k}\left(v\right)\right)$, and these $I_{k}\left(N_{k}\left(v\right)\right)$
are all disjoint.}
\end{figure}

More formally, let $s,\ell>0$ be integers, and consider a subset
$S_{\ell}\left(x\right)$ of the sphere of radius $\ell$ around $x$,
defined as follows:
\begin{align*}
S_{0}\left(x\right) & =\left\{ \left(x,n-1\right)\right\} \\
S_{i}\left(x\right) & =\bigcup_{\left(z,t\right)\in S_{i-1}\left(x\right)}M\left(z,s,t-1\right).
\end{align*}
By the remark above, the sets $\left\{ I_{s}\left(z\right)\right\} _{\left(z,t\right)\in S_{\ell}\left(x\right)}$
are all disjoint, and so the set $U\left(x\right):=\union_{\left(z,t\right)\in S_{\ell}\left(x\right)}I_{s}\left(z\right)$
has cardinality $2^{s}\abs{S_{\ell}\left(x\right)}$. Define $S_{\ell}\left(y\right)$
and $U\left(y\right)$ similarly. Write the values of $s$ and $\ell$
as $s=\frac{n}{2}-\frac{1}{2}\alpha\left(n\right)$ and $\ell=\frac{n}{\beta\left(n\right)}$,
for some functions $\alpha,\beta:\n\to\n$ to be chosen later. Assuming
that there is a vertex $u\in U\left(x\right)$ which is connected
to $v\in U\left(y\right)$, the distance between $x$ and $y$ can
be bounded as follows:
\begin{equation}
d_{\boldsymbol{G}_{n}}\left(x,y\right)\leq2\ell+D\left(\boldsymbol{G}_{s}^{u}\right)+D\left(\boldsymbol{G}_{s}^{v}\right)+1,\label{eq:partial_diameter_bound}
\end{equation}
where $\ell$ bounds the distance to go from $x$ to a vertex $z$
in $S_{\ell}\left(x\right)$, $D\left(\boldsymbol{G}_{s}^{u}\right)$
bounds the distance from $z$ to $u$, and $1$ is the distance from
$u$ to $v$ (see Figure \ref{fig:the_path}).

\begin{figure}[H]
\begin{centering}
\includegraphics[scale=0.33]{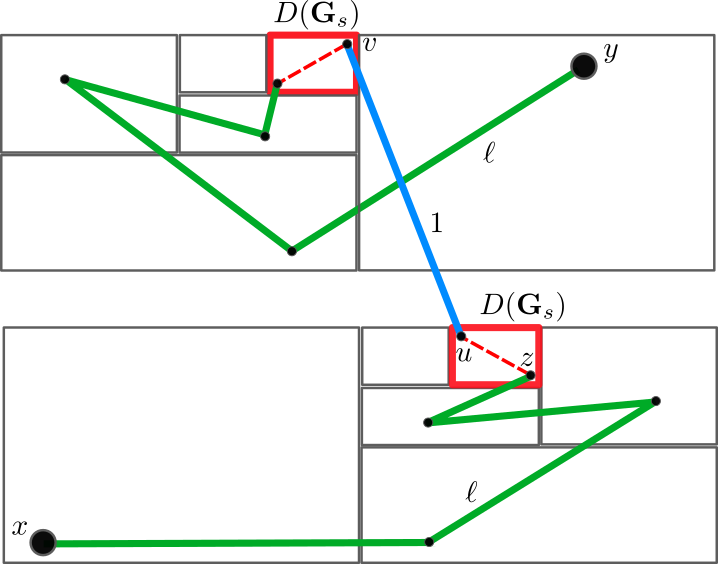}
\par\end{centering}
\caption{\label{fig:the_path}If $U\left(x\right)$ is connected to $U\left(y\right)$,
we have a path from $x$ to $y$. The red dotted lines represent an
optimal path within $\boldsymbol{G}_{s}$.}
\end{figure}
We now analyze $D\left(\boldsymbol{G}_{s}^{u}\right)+D\left(\boldsymbol{G}_{s}^{v}\right)$.
By choice of $s$, we have 
\begin{align*}
2C\frac{s}{\log\log s} & =2C\frac{\frac{n}{2}-\frac{1}{2}\alpha\left(n\right)}{\log\log\left(\frac{n}{2}-\frac{1}{2}\alpha\left(n\right)\right)}\overset{\left(\text{assume }\alpha\left(n\right)\leq\frac{1}{2}n\right)}{\leq}C\frac{n-\alpha\left(n\right)}{\log\log\left(\frac{n}{4}\right)}=C\frac{n\left(1-\frac{\alpha\left(n\right)}{n}\right)}{\log\log n+\log\left(1-\frac{\log4}{\log n}\right)}\\
 & \leq C\frac{n\left(1-\frac{\alpha\left(n\right)}{n}\right)}{\log\log n-\frac{2\log4}{\log n}}=C\frac{n}{\log\log n}\frac{1-\frac{\alpha\left(n\right)}{n}}{1-\frac{2\log4}{\log n\log\log n}}\leq C\frac{n}{\log\log n}\left(1-\frac{\alpha\left(n\right)}{n}\right)\left(1+\frac{6}{\log n\log\log n}\right)\\
 & \leq C\frac{n}{\log\log n}\left(1+\frac{6}{\log n\log\log n}-\frac{\alpha\left(n\right)}{n}\right).
\end{align*}
Since we assume that $E_{s}^{u}$ and $E_{s}^{v}$ hold, i.e. that
$D\left(\boldsymbol{G}_{s}^{v}\right),D\left(\boldsymbol{G}_{s}^{u}\right)\leq Cs\frac{\log\log\log s}{\log\log s}\leq Cs\frac{\log\log\log n}{\log\log s}$,
we have
\begin{align*}
D\left(\boldsymbol{G}_{s}^{u}\right)+D\left(\boldsymbol{G}_{s}^{v}\right) & \leq Cn\frac{\log\log\log n}{\log\log n}\left(1+\frac{6}{\log n\log\log n}-\frac{\alpha\left(n\right)}{n}\right).
\end{align*}
Choosing $\alpha\left(n\right)=\frac{17n}{\log n}$ then gives 
\[
D\left(\boldsymbol{G}_{s}^{u}\right)+D\left(\boldsymbol{G}_{s}^{v}\right)\leq Cn\frac{\log\log\log n}{\log\log n}\left(1-\frac{1}{\log n}\right).
\]
Choosing also $\beta\left(n\right)=\frac{\log2}{18}\log n\log\log n$,
so that $\ell=\frac{18}{\log2}\frac{n}{\log n\log\log n}$, by (\ref{eq:partial_diameter_bound})
we have that 
\begin{align*}
d_{\boldsymbol{G}_{n}}\left(x,y\right) & \leq\frac{36}{\log2}\frac{n}{\log n\log\log n}+Cn\frac{\log\log\log n}{\log\log n}\left(1-\frac{1}{\log n}\right)+1\\
\left(\text{for \ensuremath{C} large enough}\right) & \leq Cn\frac{\log\log\log n}{\log\log n}.
\end{align*}
All that remains is to bound the probability of the event $\left\{ U\left(x\right)\sim U\left(y\right)\all x,y\in V_{n}\right\} $
from below. We do this using a union bound. The number of vertices
in $S_{\ell}\left(x\right)$ can readily be seen to be 
\[
\sum_{k_{1}=s}^{n-1}\sum_{k_{2}=s}^{k_{1}-1}\sum_{k_{3}=s}^{k_{2}-1}\ldots\sum_{k_{\ell}=s}^{k_{\ell-1}-1}1=\sum_{k_{1}=1}^{n-s-1}\sum_{k_{2}=1}^{k_{1}-1}\sum_{k_{3}=1}^{k_{2}-1}\ldots\sum_{k_{\ell}=1}^{k_{\ell-1}-1}1.
\]
This is the number of decreasing positive integer sequences of length
$\ell$, whose maximum entry is bounded by $n-s-1$. Since every choice
of $\ell$ integers can be ordered in a unique fashion, we have 
\begin{align*}
\abs{S_{\ell}\left(x\right)} & ={n-s-1 \choose \ell}\\
 & ={\frac{n}{2}+\frac{1}{2}\alpha\left(n\right)-1 \choose n/\beta\left(n\right)}\\
 & \geq\left(\frac{\frac{n}{2}+\frac{1}{2}\alpha\left(n\right)-1}{n/\beta\left(n\right)}\right)^{n/\beta\left(n\right)}\\
 & \geq\left(\frac{\beta\left(n\right)}{2}\right)^{n/\beta\left(n\right)}\\
 & =\exp\left(\left(\log\beta\left(n\right)-\log2\right)\frac{n}{\beta\left(n\right)}\right)\\
\left(\text{Assume \ensuremath{n_{0}} large so that \ensuremath{\log\beta\left(n\right)\geq2\log2}}\right) & \geq\exp\left(\frac{n\log\beta\left(n\right)}{2\beta\left(n\right)}\right).
\end{align*}
The collection $U\left(x\right)=\union_{\left(z,t\right)\in S_{\ell}\left(x\right)}I_{s}\left(z\right)$
has size at least
\[
\abs{U\left(x\right)}=2^{s}\abs{S_{\ell}\left(x\right)}\geq2^{\frac{n}{2}-\frac{1}{2}\alpha\left(n\right)+\frac{\log2}{2}\frac{n\log\beta\left(n\right)}{\beta\left(n\right)}}.
\]
Denoting $U=\abs{U\left(x\right)}=\abs{U\left(y\right)}$, the probability
that the sets $U\left(x\right)$ and $U\left(y\right)$ are disconnected
at the $n$-th step is therefore bounded from above by
\begin{align*}
\p\left[U\left(x\right)\not\sim U\left(y\right)\right] & =\frac{{2^{n-1}-U \choose U}}{{2^{n-1} \choose U}}\\
 & =\frac{\left(2^{n-1}-U\right)\left(2^{n-1}-U-1\right)\cdots\left(2^{n-1}-2U+1\right)}{2^{n-1}\left(2^{n-1}-1\right)\cdots\left(2^{n-1}-U+1\right)}\\
\left(\text{AM-GM inequality}\right) & \leq\frac{\left(2^{n-1}-\frac{3U-1}{2}\right)^{U}}{\left(2^{n-1}-U\right)^{U}}\\
 & =\left(1-\frac{1}{2}\frac{U-1}{2^{n-1}-U}\right)^{U}\\
 & \leq\left(1-\frac{U-1}{2^{n}}\right)^{U}\\
 & \leq\exp\left(-U^{2}/2^{n}+U/2^{n}\right)\\
 & \leq\exp\left(-2^{-\alpha\left(n\right)+\log2\frac{n\log\beta\left(n\right)}{\beta\left(n\right)}+o\left(1\right)}\right).
\end{align*}
Plugging in our choice of $\alpha\left(n\right)$ and $\beta\left(n\right)$,
we get 

\begin{align*}
\p\left[U\left(x\right)\not\sim U\left(y\right)\right] & \leq\exp\left(-2^{-\frac{17n}{\log n}+\log2\frac{n\log\left(\frac{\log2}{18}\log n\log\log n\right)}{\frac{\log2}{18}\log n\log\log n}+o\left(1\right)}\right)\\
 & \leq\exp\left(-2^{-\frac{17n}{\log n}+\frac{18n}{\log n}\left(1+o\left(1\right)\right)}\right)\\
 & =\exp\left(-2^{\frac{n}{\log n}\left(1+o\left(1\right)\right)}\right)\\
 & \leq\exp\left(-2^{\frac{n}{2\log n}}\right)
\end{align*}
for $n$ large enough. As there are no more than $2^{2n}=e^{2n\log2}$
choices for the pairs $x,y$, this gives 
\[
\p\left[\exists x,y\text{ s.t. }U\left(x\right)\not\sim U\left(y\right)\right]\leq\exp\left(-2^{\frac{n}{2\log n}}+2n\log2\right)\leq\exp\left(-2^{\frac{n}{4\log n}}\right)
\]
for $n$ large enough. Let $E_{k}=\union_{z\in\left\{ 0,1\right\} ^{n}}E_{k}^{z}$.
We have thus shown that 
\[
\p\left[E_{n}\mid E_{1},\ldots,E_{n-1}\right]\geq1-\exp\left(-2^{\frac{n}{4\log n}}\right).
\]
Since there are $2^{n-k}$ instances of $\boldsymbol{G}_{k}$ in $\boldsymbol{G}_{n}$,
and recalling that $\p\left[E_{k}\right]=1$ for $k\leq k^{*}$, by
repeated conditioning we thus have 
\begin{align*}
\p\left[E_{n}\right] & \geq1-\sum_{k=k^{*}}^{n}\exp\left(-2^{\frac{k}{4\log k}}\right)2^{n-k}\\
 & \geq1-\sum_{k=k^{*}}^{n}\exp\left(-2^{\frac{k}{4\log k}}+n\log2\right).
\end{align*}
By choosing $\gamma$ large enough, for $k\geq k^{*}$ we have 
\begin{align*}
\frac{k}{4\log k} & \geq\frac{k^{*}}{4\log k^{*}}=\frac{\gamma\log n\log\log n}{4\log\left(\gamma\log n\log\log n\right)}\\
 & \geq2\log_{2}n,
\end{align*}
and so 
\[
\p\left[E_{n}\right]\geq1-\sum_{k=k^{*}}^{n}e^{-n}.
\]
\end{proof}

\subsection{Vertex expansion\label{subsec:postponed_expansion_proofs}}

The proof of Theorem \ref{thm:good_vertex_expander} relies on the
observation that a set $S\subseteq V_{n}$ sampled uniformly at random
will be an $\alpha$-vertex expander with high probability (for some
small constant $\alpha>0$), since a constant fraction of the edges
of $\sigma_{n-1}$ will go from $S$ to its complement. This alone
is not enough, since there are always sets of the form $S=S_{0}\union S_{1}$,
where $S_{0}\subseteq V_{n}^{0}$ (recall (\ref{eq:disjoint_copy_definition})
for the definition of $V_{n}^{z}$) and $S_{1}=\left\{ N_{n}\left(x\right)\mid x\in S_{0}\right\} $.
To overcome this, we look at edges coming from the last three permutations,
$\sigma_{n-1},\sigma_{n-2},\sigma_{n-3}$, and bound the number of
sets $S\subseteq V_{n}$ so that the boundary that comes from $\sigma_{n-3}$-edges
isn't large enough. Afterwards we apply a union bound over these sets
to bound the probability that they have a small $\sigma_{n-1}$- and
$\sigma_{n-2}$-boundary.

More precisely, sets which have a small contribution to their boundary
at the $k$-th generation are defined as follows.
\begin{defn}[Badly-matched sets]
Let $x\in\left\{ 0,1\right\} ^{k}$. Let $A\subseteq V_{k}^{0}$,
$B\subseteq V_{k}^{1}$. We say that $A,B$ are \emph{$\left(k,\alpha\right)$-badly-matched}
if 
\[
\frac{2\abs{x\in A\mid N_{k}\left(x\right)\in B}}{\abs A+\abs B}\geq\left(1-\alpha\right).
\]
\end{defn}

\begin{rem}
\label{rem:badly_matched_sets} If $A,B$ are badly-matched, then
$\abs{\abs A-\abs B}\leq\alpha\left(\abs A+\abs B\right)$. This is
because if, say, $\abs A>\abs B+\alpha\left(\abs A+\abs B\right)$
then even when all edges from $B$ go into \textbf{$A$ }there will
still be $\alpha\left(\abs A+\abs B\right)$ edges between $A$ and
$V_{k}^{1}\setminus B$. This implies that 

\[
2\abs{x\in A\mid N_{k}\left(x\right)\in B}\leq2\abs B<\left(1-\alpha\right)\left(\abs A+\abs B\right).
\]
If $A,B$ are not badly-matched, then $\abs{\partial_{k}\left(A\union B\right)}>\alpha\abs{A\union B}$,
since $\abs{\partial_{k}\left(A\union B\right)}=\abs A+\abs B-2\abs{x\in A\mid N_{k}\left(x\right)\in B}$,
so the set $A\union B$ has $\alpha$-expansion. If $A,B$ are $\left(k,\alpha\right)$-badly-matched,
then they are also $\left(k,\alpha'\right)$-badly-matched for every
$\alpha'\ge\alpha$. 
\end{rem}

As alluded to above, we start by bounding the possible number of badly-matched
sets in generation $n-2$, for any permutation $\sigma_{n-3}$; this
is the content of Proposition \ref{prop:number_of_badly_matched_pairs}.
We then bound the probability that said badly-matched sets are also
badly-matched in generations $n-1$ and $n$; this is the content
of Proposition \ref{prop:probability_of_badly_matched}. The last
claim we need for the proof is that sets of size $O(n)$ have non-trivial
vertex expansion regardless of the permutation. The proofs of all
assertions are found at the end of the section. 

Recall that for an integer $k>0$, we set $K:=2^{k}$. In addition,
denote by $H(x)=-x\log x-(1-x)\log(1-x)$ the binary entropy function.
\begin{prop}
\label{prop:number_of_badly_matched_pairs}There exists a function
$\delta:\r\to\r$ with $\lim_{x\to0}\delta\left(x\right)=0$, that
depends on $\eta$, such that the following holds. Let $\alpha>0$,
and let $k,j>0$ be integers so that $j\leq\eta K$. For any permutation
$\sigma_{k-1}$, the number of $\left(k,\alpha\right)$-badly-matched
sets $A\subseteq V_{k}^{0}$ and $B\subseteq V_{k}^{1}$ such that
$(1-\alpha)\frac{j}{2}\leq\abs A,\abs B\leq(1+\alpha)\frac{j}{2}$
is smaller than 
\[
5\alpha^{3}K^{3}2^{\frac{K}{2}(1+\delta(\alpha))H(\frac{j}{K})+j\delta(\alpha)}.
\]
\end{prop}

\begin{prop}
\label{prop:probability_of_badly_matched} There exists a function
$\delta:\r\to\r$ with $\lim_{x\to0}\delta\left(x\right)=0$ that
depends on $\eta$, such that the following holds. Let $\alpha>0$,
and let $k,j>0$ be integers so that $j\leq\eta K$. Let $A\subseteq V_{k}^{0}$
and $B\subseteq V_{k}^{1}$ be such that $(1-\alpha)\frac{j}{2}\leq\abs A,\abs B\leq(1+\alpha)\frac{j}{2}$.
If the permutations $\permvector$ are uniformly random, then 
\begin{equation}
\p\left[A,B\text{ are \ensuremath{\left(k,\alpha\right)}-badly-matched}\right]\leq3\alpha K^{2}2^{-\frac{K}{2}(1-\delta(\alpha))H(\frac{j}{K})+\delta(\alpha)j}.\label{eq:badly_matched_probability_bound}
\end{equation}
\end{prop}

\begin{claim}
\label{claim:small-sets-vertex-expand}Let $c>3$. Then there is some
$n_{0}\in\mathbb{N}$ so that for every $n>n_{0}$ and every $S\subseteq V_{n}$
so that $|S|\leq cn$, 
\[
|\partial S|\geq\frac{1}{c^{2}}|S|.
\]
\end{claim}

\begin{proof}[Proof of Theorem \ref{thm:good_vertex_expander}]
 Fix $\eta>0$ and let $\alpha>0$ be chosen later. We define $F_{n}$
to be the event that $\boldsymbol{G}_{n}$ is not an $(\eta,\alpha)$-vertex
expander. We will show that 
\[
\lim_{n\to\infty}\p[F_{n}]=0.
\]
 We will bound
\begin{equation}
\p\left[F_{n}\right]\leq\sum_{j=0}^{\eta N}\p\left[\exists S\text{ s.t. }\abs S=j\text{ and }\abs{\partial S}<\alpha\abs S\right]\label{eq:vertex-expansion-union-bound}
\end{equation}
By Claim \ref{claim:small-sets-vertex-expand}, it is enough to start
this sum with $j=cn$, as long as $\alpha\leq\frac{1}{c^{2}}.$The
constant $c$ will be determined at the end of the proof. Thus the
right-hand side of (\ref{eq:vertex-expansion-union-bound}) is equal
to
\[
\sum_{j=cn}^{\eta N}\p\left[\exists S\text{ s.t. }\abs S=j\text{ and }\abs{\partial S}<\alpha|S|\right].
\]
Let $S\subseteq V_{n}$ with $\abs S=j$. For $x\in\left\{ 0,1\right\} ^{k}$,
denote $S_{x}=S\intersect V_{n}^{x}$. If the sets $S_{0},S_{1}$
are not $\left(n,\alpha\right)$-badly-matched, then by Remark \ref{rem:badly_matched_sets},
the edges from $\sigma_{n-1}$ are enough to guarantee a large boundary,
i.e. the set $S$ has $\alpha$-expansion. This happens in particular
when $\abs{\abs{S_{1}}-\abs{S_{0}}}\geq\alpha\abs S$. We may thus
restrict ourselves to $S$ that satisfy $\abs{S_{i}}\geq\frac{1}{2}\left(1-\alpha\right)\abs S$.
Thus

\begin{equation}
\left(1-\alpha\right)\frac{j}{2}\leq\abs{S_{i}}\leq\left(1+\alpha\right)\frac{j}{2}\label{eq:si_nearly_balanced}
\end{equation}
for every $i\in\left\{ 0,1\right\} $. Similarly, if $S_{i0,},S_{i1}$
are not $\left(n-1,3\alpha\right)$-badly-matched for any $i\in\left\{ 0,1\right\} $,
then $\abs{\partial_{n-1}\left(S_{i0}\union S_{i1}\right)}\geq3\alpha\abs{S_{i}}\geq\frac{3}{2}\alpha\left(1-\alpha\right)\abs S$,
which is larger than $\alpha\abs S$ for $\alpha$ small enough. This
happens in particular when $\abs{\abs{S_{i1}}-\abs{S_{i0}}}\geq3\alpha\abs{S_{i}}$.
We may thus further restrict ourselves to $S$ that satisfy $\abs{S_{ij}}\geq\frac{1}{2}\left(1-3\alpha\right)\abs{S_{i}}$,
which means that 
\begin{equation}
(1-3\alpha)\left(1-\alpha\right)\frac{j}{4}\leq\abs{S_{ij}}\leq(1+3\alpha)\left(1+\alpha\right)\frac{j}{4}\label{eq:sij_nearly_balanced}
\end{equation}
for every $i,j\in\left\{ 0,1\right\} $. 

Finally, if there is are sets $S_{ij0},S_{ij1}$ for some $i,j\in\left\{ 0,1\right\} $
that are not $\left(n-2,5\alpha\right)$-badly-matched, then for $\alpha$
small enough we have $\abs{\partial_{n-2}\left(S_{ij0}\union S_{ij1}\right)}\geq\alpha\abs S$,
and we can assume that 

\[
(1-5\alpha)(1-3\alpha)\left(1-\alpha\right)\frac{j}{8}\leq\abs{S_{ij}}\leq(1+5\alpha)(1+3\alpha)\left(1+\alpha\right)\frac{j}{8}.
\]
In particular, this happens when
\[
(1-10\alpha)\frac{j}{8}\leq\abs{S_{ij}}\leq(1+10\alpha)\frac{j}{8}.
\]
Thus, to bound $\p\left[\exists S\text{ s.t. }\abs S=j\text{ and }\abs{\partial S}<\alpha\abs S\right],$
we only need to consider sets $S$ whose all four pairs $S_{ij0},S_{ij1}$
are $\left(n-2,10\alpha\right)$ badly-matched; any other set has
$\alpha$-expansion by Remark \ref{rem:badly_matched_sets}. By Proposition
\ref{prop:number_of_badly_matched_pairs}, with $k=n-2$, for each
$i,j\in\left\{ 0,1\right\} $ there are at most $5000\alpha^{3}N^{3}2^{\frac{N}{8}\left(1+\delta\left(10\alpha\right)\right)H\left(\frac{j}{N}\right)+j\delta\left(10\alpha\right)}$
sets $S_{ij}$ such that $S_{ij0},S_{ij1}$ are $\left(n-2,10\alpha\right)$-badly-matched,
so there are at most 
\[
\left(5000\alpha^{3}N^{3}2^{\frac{N}{8}\left(1+\delta\left(10\alpha\right)\right)H\left(\frac{j}{N}\right)+j\delta\left(10\alpha\right)}\right)^{4}=5000^{4}\alpha^{12}N^{12}2^{\frac{N}{2}\left(1+\delta\left(10\alpha\right)\right)H\left(\frac{j}{N}\right)+4j\delta\left(10\alpha\right)}
\]
possible sets to consider. Thus 
\begin{align}
\p\left[\exists S\text{ s.t. }\abs S=j\text{ and }\abs{\partial S}<\alpha|S|\right] & \leq5000^{4}\alpha^{12}N^{12}2^{\frac{N}{2}\left(1+\delta\left(10\alpha\right)\right)H\left(\frac{j}{N}\right)+4j\delta\left(10\alpha\right)}\cdot\label{eq:less_summands_in_union_bound}\\
 & \,\,\,\,\,\,\,\,\max_{S}\p\left[S\text{ does not have \ensuremath{\alpha}-expansion}\right],
\end{align}
where $S$ is restricted as above. To bound the probability, observe
that for any fixed $S\subseteq V_{n}$, 
\begin{align*}
\p\left[S\text{ does not have \ensuremath{\alpha}-expansion}\right] & \leq\p[S_{0},S_{1}\text{ are \ensuremath{\left(n,\alpha\right)}-badly-matched }\\
 & \text{\,\,\,\,\,\,\,\,\,\,\,\,\,\,\,\,\,\,\,and }S_{00},S_{01}\text{ are \ensuremath{\left(n-1,3\alpha\right)}-badly-matched}].
\end{align*}
As the event $\left\{ S_{0},S_{1}\text{ are badly-matched}\right\} $
depends only on the permutation $\sigma_{n-1}$ and $\left\{ S_{00},S_{11}\text{ are badly-matched}\right\} $
depends only on the permutation $\sigma_{n-2}$, these two events
are independent. By the relations (\ref{eq:si_nearly_balanced}) and
(\ref{eq:sij_nearly_balanced}), we can apply Proposition \ref{prop:probability_of_badly_matched},
yielding
\begin{align*}
\p\left[S\text{ does not have \ensuremath{\alpha}-expansion}\right] & =\p\left[S_{0},S_{1}\text{ are not \ensuremath{\left(n,\alpha\right)}-badly-matched}\right]\\
 & \,\,\,\,\,\,\,\,\,\,\,\,\cdot\p\left[S_{00},S_{01}\text{ are not \ensuremath{\left(n-1,3\alpha\right)}-badly-matched}\right].\\
 & \leq3\alpha N^{2}2^{-\frac{N}{2}\left(1-\delta\left(\alpha\right)\right)H\left(\frac{j}{N}\right)+\delta\left(\alpha\right)j}\cdot3\alpha N^{2}2^{-\frac{N}{4}\left(1-\delta\left(3\alpha\right)\right)H\left(\frac{j}{N}\right)+\delta\left(3\alpha\right)j}.
\end{align*}
For simplicity, in the next inequalities we unify all the expressions
of the form $\delta(c\alpha)$ appearing in the exponents to $\delta(\alpha)$
(that goes to $0$ as $\alpha\to0$). Using (\ref{eq:less_summands_in_union_bound}),
we get that, 
\[
\p\left[\exists S\text{ s.t. }\abs S=j\text{ and }\abs{\partial S}<\alpha|S|\right]\leq9\cdot5000^{4}\alpha^{14}N^{16}2^{\delta\left(\alpha\right)j}(2^{\frac{1}{4}\left(1-\delta\left(\alpha\right)\right)})^{-NH\left(\frac{j}{N}\right)}.
\]
Note that we abused notation. Plugging this back in (\ref{eq:vertex-expansion-union-bound})
we get
\begin{equation}
\p\left[F_{n}\right]\leq9\cdot5000^{4}\alpha^{14}N^{16}\sum_{j=cn}^{\eta N}2^{\delta\left(\alpha\right)j}\left(2^{\frac{1}{4}(1-\delta(\alpha))}\right)^{-NH\left(\frac{j}{N}\right)}.\label{eq:midway_calc_expansion}
\end{equation}
We take $\alpha$ so that $\delta\left(\alpha\right)<\frac{1}{2}$
and get that $2^{\frac{1}{4}(1-\delta(\alpha))}\geq2^{\frac{1}{8}}$.
In addition, we use the well known inequality $H\left(x\right)\geq4x\left(1-x\right)$
to bound the right-hand side of (\ref{eq:midway_calc_expansion})
from above by

\[
9\cdot5000^{4}\alpha^{14}N^{16}\sum_{j=cn}^{\eta N}2^{\delta\left(\alpha\right)j}2{}^{-\frac{j}{2}\left(1-\eta\right)}=9\cdot5000^{4}\alpha^{14}N^{16}\sum_{j=cn}^{\eta N}2{}^{-\frac{j}{2}\left(1-\eta-2\delta\left(\alpha\right)\right)}.
\]
Finally, by taking $\alpha$ so that $1-\eta-2\delta(\alpha)\geq\frac{1-\eta}{2}$
and taking $c$ so that $2^{\frac{cn/2}{2}(\frac{1-\eta}{2})}\geq N^{16}=2^{16n}$
we get that the sum on the right-hand side is at most
\begin{align*}
9\cdot5000^{4}\alpha^{14}N^{16}2^{-\frac{cn/2}{2}(\frac{1-\eta}{2})}\sum_{j=\frac{c}{2}n}^{\infty}2{}^{-\frac{j}{4}\left(1-\eta\right)} & \leq9\cdot5000^{4}\alpha^{14}\sum_{j=\frac{c}{2}n}^{\infty}2{}^{-\frac{j}{4}(1-\eta)}\\
 & =\frac{9}{1-2^{-(1-\eta)/4}}\cdot5000^{4}\alpha^{14}2^{-\frac{cn}{8}(1-\eta)}.
\end{align*}
Thus,
\[
\p\left[F_{n}\right]\leq\frac{9}{1-2^{-(1-\eta)/4}}\cdot5000^{4}\alpha^{14}2^{-\frac{cn}{8}(1-\eta)}.
\]
This tends to $0$ as $n\to\infty$.
\end{proof}
\begin{proof}[Proof of Proposition \ref{prop:number_of_badly_matched_pairs}]
We can count the subsets $A\subseteq V_{k}^{0},B\subseteq V_{k}^{1}$
by first choosing a set of edges of the $k$-th matching that are
connected to $A\cup B$. For each chosen edge $\{x,y\}$ where $x\in V_{k}^{0},y\in V_{k}^{1}$
we decide whether $x\in A,y\in B$, or $x\in A,y\notin B$ or $x\notin A,y\in B$.
For sets that are $(k,\alpha)$-badly matched, our count yields the
following. 
\begin{enumerate}
\item The number of edges that are adjacent to $A\cup B$ is at least $(1-\alpha)\frac{j}{2}$
(the lower bound is achieved when $\abs A=\abs B=(1-\alpha)\frac{j}{2}$
and $N_{k}\left(A\right)=B$). It is at most $\left(1+\alpha\right)\frac{j}{2}+2\alpha\left(1+\alpha\right)\frac{j}{2}$
(since there could be at most $(1+\alpha)\frac{j}{2}$ that cross
from $A$ to $B$, and no more than $2\alpha(1+\alpha)\frac{j}{2}$
additional edges that are adjacent to only one of $A,B$, since $A,B$
are supposed to be $(k,\alpha)$-badly matched). Since $\left(1+\alpha\right)\frac{j}{2}+2\alpha\left(1+\alpha\right)\frac{j}{2}\leq\left(1+4\alpha\right)\frac{j}{2}$
for $\alpha\leq\frac{1}{2}$, using the relation ${n \choose k}\leq2^{nH\left(k/n\right)}$,
the number of possible choices for edges adjacent to $A\cup B$ is
at most
\begin{align}
\sum_{\ell=(1-\alpha)\frac{j}{2}}^{(1+4\alpha)\frac{j}{2}}{\frac{K}{2} \choose \ell} & \leq\sum_{\ell=(1-\alpha)\frac{j}{2}}^{(1+4\alpha)\frac{j}{2}}2^{\frac{K}{2}H(\frac{\ell}{K/2})}\leq\sum_{\ell=(1-\alpha)\frac{j}{2}}^{(1+4\alpha)\frac{j}{2}}2^{\frac{K}{2}H\left(\frac{j}{K}\right)+\frac{K}{2}\left(H\left(\frac{\ell}{K/2}\right)-H\left(\frac{j/2}{K/2}\right)\right)}.\label{eq:num-of-edges-adjacent-to-union}
\end{align}
By Lagrange's mean-value theorem, $\frac{K}{2}\left(H\left(\frac{\ell}{K/2}\right)-H\left(\frac{j/2}{K/2}\right)\right)=\left(\ell-\frac{j}{2}\right)H'\left(\xi\right)=-\left(\ell-\frac{j}{2}\right)\log\frac{\xi}{1-\xi}$
for some $\xi$ between $\frac{\ell}{K/2}$ and $\frac{j/2}{K/2}$.
Thus we can write $\xi=c'\frac{j/2}{K/2}$ for some $(1-\alpha)\leq c'\leq(1+4\alpha)$,
and we have 
\begin{align}
\abs{\left(\ell-\frac{j}{2}\right)H'\left(\xi\right)} & \leq4\alpha j\abs{H'\left(\xi\right)}\nonumber \\
 & =4\alpha j\abs{\log\xi-\log\left(1-\xi\right)}.\label{eq:entropy_derivative}
\end{align}
We now bound the logarithms. Since $\xi=c'\frac{j}{K}\leq c'\eta$,
we have that $-\log\left(1-\xi\right)\leq-\log\left(1-c'\eta\right)$;
the quantity on the right-hand side is just a constant (provided that
$\alpha$ is small enough so that $\left(1+4\alpha\right)\eta<1$).
For $\abs{\log\xi}$, we have 
\begin{align*}
4\alpha j\abs{\log\xi} & =4\alpha j\abs{\log\left(c'\frac{j}{K}\right)}\leq4\alpha j\abs{\log c'}+4\alpha j\log\frac{j}{K}\\
 & =4\alpha j\abs{\log c'}+8\frac{K}{2}\alpha\frac{j}{K}\log\frac{j}{K}\\
 & \leq4\alpha j\abs{\log c'}+8\frac{K}{2}\alpha H\left(\frac{j}{K}\right).
\end{align*}
Thus, there exists a constant $c>0$ (that depends on $\eta$) such
that 
\[
\abs{\left(\ell-\frac{j}{2}\right)H'\left(\xi\right)}\leq c\alpha j+c\alpha\frac{K}{2}H\left(\frac{j}{K}\right).
\]
for some $c>0$ that depends on $\eta$. Thus the left-hand side in
(\ref{eq:num-of-edges-adjacent-to-union}) is at most
\begin{equation}
5\alpha K2^{(1+c\alpha)\frac{K}{2}H(\frac{j}{K})+c\alpha j}.\label{eq:final-num-of-edges-adjacent-to-union}
\end{equation}
\item Then we choose out of the edges adjacent to $A\cup B$ the edges that
touch $A$ only, and the edges that touch $B$ only. As $A,B$ are
$(k,\alpha)$-badly matched, at least a $(1-\alpha)$-fraction of
the edges must touch both $A$ and $B$, so no more than an $\alpha$-fraction
of the edges are available to touch only one of the sets. Assuming
that $\alpha<1/4$, the number of possibilities (for a given edge
set chosen in the previous step) is at most
\begin{equation}
\sum_{\ell_{A}=0}^{\alpha(1+4\alpha)\frac{j}{2}}\binom{(1+4\alpha)\frac{j}{2}}{\ell_{A}}\sum_{\ell_{B}=0}^{\alpha(1+4\alpha)\frac{j}{2}}\binom{(1+4\alpha)\frac{j}{2}}{\ell_{B}}\leq(\alpha j)^{2}\binom{(1+4\alpha)\frac{j}{2}}{\alpha(1+4\alpha)\frac{j}{2}}^{2}\leq\alpha^{2}K^{2}2^{2jH(\alpha)}.\label{eq:only_one_endpoint_in_A_or_B}
\end{equation}
\end{enumerate}
Multiplying (\ref{eq:final-num-of-edges-adjacent-to-union}) and (\ref{eq:only_one_endpoint_in_A_or_B}),
and setting $\delta(\alpha)=2H(\alpha)+c\alpha$, the number of badly-matched
pairs is bounded by 
\[
5\alpha^{3}K^{3}2^{\frac{K}{2}(1+\delta(\alpha))H(\frac{j}{K})+\delta(\alpha)j}.
\]

\end{proof}
\begin{proof}[Proof of Proposition \ref{prop:probability_of_badly_matched}]
To bound the probability in (\ref{eq:badly_matched_probability_bound}),
we go over all possible subsets $A'\subseteq A$ and sum the probability
that the set of outgoing edges from $A'$ is some set $B'\subseteq B$.
Since $A$ and $B$ both have sizes in the interval$\left[(1-\alpha)\frac{j}{2},(1+\alpha)\frac{j}{2}\right]$,
the size of $A',B'$ should be at least $(1-\alpha)(1-\alpha)\frac{j}{2}\geq(1-2\alpha)\frac{j}{2}.$
The probability is bounded by 
\begin{align}
\sum_{\ell=(1-2\alpha)\frac{j}{2}}^{(1+\alpha)\frac{j}{2}}\,\sum_{\stackrel{A'\subseteq A,B'\subseteq B}{|A'|=|B'|=\ell}}\p\left[N_{k}(A')=B'\right] & =\sum_{\ell=(1-2\alpha)\frac{j}{2}}^{(1+\alpha)\frac{j}{2}}\,\sum_{\stackrel{A'\subseteq A,B'\subseteq B}{|A'|=|B'|=\ell}}\frac{1}{\binom{\frac{K}{2}}{\ell}}\nonumber \\
 & \leq\sum_{\ell=(1-2\alpha)\frac{j}{2}}^{(1+\alpha)\frac{j}{2}}\,\sum_{\stackrel{A'\subseteq A,B'\subseteq B}{|A'|=|B'|=\ell}}\frac{K}{2}2^{-\frac{K}{2}H(\frac{\ell}{K/2})}\nonumber \\
 & \stackrel{\left(\text{assuming \ensuremath{\alpha<\frac{1}{5}}}\right)}{\leq}\sum_{\ell=(1-2\alpha)\frac{j}{2}}^{(1+\alpha)\frac{j}{2}}\,\binom{(1+\alpha)\frac{j}{2}}{(1-2\alpha)\frac{j}{2}}^{2}\frac{K}{2}2^{-\frac{K}{2}H(\frac{\ell}{K/2})}\nonumber \\
 & \leq2^{2(1+\alpha)\frac{j}{2}H(\frac{1-2\alpha}{1+\alpha})}\frac{K}{2}\sum_{\ell=(1-2\alpha)\frac{j}{2}}^{(1+\alpha)\frac{j}{2}}2^{-\frac{K}{2}H(\frac{\ell}{K/2})}.\label{eq:prob_of_badly_matched_pair}
\end{align}
By Lagrange's mean-value theorem, we write 
\[
\frac{K}{2}H\left(\frac{\ell}{K/2}\right)=\frac{K}{2}H\left(\frac{j}{K}\right)+\frac{K}{2}\left(H\left(\frac{\ell}{K/2}\right)-H\left(\frac{j/2}{K/2}\right)\right)=\frac{K}{2}H\left(\frac{j}{K}\right)+\left(\ell-\frac{j}{2}\right)H'\left(\xi\right)
\]
for some $\xi$ between $\frac{\ell}{K/2}$ and $\frac{j/2}{K/2}.$
As $|\ell-\frac{j}{2}|\leq\alpha j$, we bound $(\ell-\frac{j}{2})H'(\xi)$
by $\alpha j|H'(\xi)$|. Write $\xi=c\frac{j}{K}$ for some $1-2\alpha\leq c'\leq1+2\alpha$.
Then (similar to (\ref{eq:entropy_derivative}) in the proof of the
previous proposition)

\[
\abs{(\ell-\frac{j}{2})H'\left(\xi\right)}\leq c\alpha j+c\alpha\frac{K}{2}H\left(\frac{j}{K}\right)
\]
for some constant $c>0$ which only depends on $\eta$. Thus (\ref{eq:prob_of_badly_matched_pair})
is at most 
\[
2^{2(1+\alpha)\frac{j}{2}H(\frac{1-2\alpha}{1+\alpha})}\frac{K}{2}\sum_{\ell=(1-2\alpha)\frac{j}{2}}^{(1+\alpha)\frac{j}{2}}2^{-\frac{K}{2}\left(1-\delta\left(\alpha\right)\right)H(\frac{j}{K})+\delta(\alpha)j}\leq3\alpha K^{2}2^{-\frac{K}{2}(1-\delta(\alpha))H(\frac{j}{K})+\delta(\alpha)j},
\]
where $\delta\left(\alpha\right)=\max\left\{ \left(1+\alpha\right)H\left(\frac{1-2\alpha}{1+\alpha}\right),c\alpha\right\} $.
\end{proof}
\begin{proof}[Proof of Claim \ref{claim:small-sets-vertex-expand}]
For every vertex $v\in V_{n}$, the second neighborhood of $v$,
$A_{n}(v):=\{u\in V_{n}\mid d(v,u)=2\}$, is of size at least $\binom{n}{2}$.
This can be seen by induction. The base case for $n=2$ is clear.
Assume without loss of generality that $v\in V_{n}^{0}$ and partition
$A_{n}\left(v\right)=\left(A_{n}\left(v\right)\cap V_{n}^{0}\right)\cup\left(A_{n}\left(v\right)\cap V_{n}^{1}\right).$
Note that in the instance of $G_{n-1}$ whose vertex set is $V_{0}$,
the second neighborhood of $v$ is $A_{n-1}\left(v\right)=A_{n}\left(v\right)\intersect V_{n}^{0}$.
Thus, by the induction hypothesis, $\abs{A_{n}(v)\cap V_{n}^{0}}\geq\binom{n-1}{2}$.
In addition, $A_{n}(v)\cap V_{n}^{1}$ contains the neighborhood of
$N_{n}(v)$ inside $V_{n}^{1}$, which is of size $n-1$. Summing
up sizes we get $\abs{A_{n}(v)}\geq\binom{n-1}{2}+\binom{n-1}{1}=\binom{n}{2}.$
Note that for the hypercube $Q_{n}$ we have strict equality. 

Now fix $S\subseteq V_{n}$ of size at most $cn$ and let $v\in S$.
If a $\frac{1}{c}$-fraction of the neighborhood of $v$ is not in
$S$ then $|\partial S|\geq\frac{1}{c}n\geq\frac{1}{c^{2}}|S|$. Otherwise,
at least $(1-\frac{1}{c})$-fraction of $v$'s neighbors are inside
$S$. Denote these vertices as $T:=N(v)\cap S$. Thus $\abs{\partial S}\geq|N(T)|-|S|$.
Since the neighborhood of the neighborhood of $v$ is $A_{n}(v)\cup\{v\}$,
the neighborhood of $T$ is of size at least 
\[
\abs{A_{n}\left(v\right)}-n\abs{N\left(v\right)\setminus S}\geq\binom{n}{2}-\frac{1}{c}n^{2}\geq\frac{1}{7}n^{2}.
\]
Hence $\abs{\partial S}\geq\frac{1}{7}n^{2}-cn\geq\frac{1}{c^{2}}|S|$
for a large enough $n$.
\end{proof}

\subsection{Eigenvalues\label{subsec:postponed_eigenvalue_proofs}}
\begin{proof}[Proof of Proposition \ref{prop:bad_gap_for_duplimatching}]
Since $G_{n}$ is $n$-regular, its largest eigenvalue $\lambda_{1}$
is $n$, and its corresponding eigenvector $f_{n}:\left\{ 0,1\right\} ^{n}\to\r$
satisfies $f_{n}\left(x\right)=1$. To show that $\lambda_{2}\geq n-2$,
let $g_{n}:\left\{ 0,1\right\} ^{n}\to\r$ be given by 
\[
g_{n}\left(x\right)=\left(-1\right)^{x_{n}},
\]
i.e. $g_{n}$ takes value $1$ on the first instance of $G_{n-1}$
in $G_{n}$, and $-1$ on the second instance. Then $A_{n}g_{n}=\left(n-2\right)g_{n}$. 

The proof that $\lambda_{2}\leq n-2$ is by induction. The claim clearly
holds for $n=1$, where $G_{1}$ is just an edge. Assume it holds
for all $k\leq n-1$, and let $h:\left\{ 0,1\right\} ^{n}\to\r$ be
an eigenvector of $A_{n}$ that is orthogonal to both $f_{n}$ and
$g_{n}$. If we write 
\[
h\left(x\right)=\begin{cases}
h_{0}\left(x_{1},\ldots,x_{n-1}\right) & x_{n}=0\\
h_{1}\left(x_{1},\ldots,x_{n-1}\right) & x_{n}=1,
\end{cases}
\]
for some functions $h_{i}:\left\{ 0,1\right\} ^{n-1}\to\r$, then
both $h_{0}$ and $h_{1}$ are orthogonal to $f_{n-1}$, and by the
induction hypothesis, we have $h_{i}^{T}A_{n-1}h_{i}\leq\left(n-3\right)\norm{h_{i}}_{2}^{2}$.
Using the recursive matrix representation (\ref{eq:matrix_next_step})
of the twisted hypercube graph, we can write 
\[
h^{T}A_{n}h=\left(\begin{array}{cc}
h_{0}^{T} & h_{1}^{T}\end{array}\right)\left(\begin{array}{cc}
A_{n-1}^{0} & P\\
P^{T} & A_{n-1}^{1}
\end{array}\right)\left(\begin{array}{c}
h_{0}\\
h_{1}
\end{array}\right),
\]
where $P$ is the $2^{n-1}\times2^{n-1}$ permutation matrix representing
$\sigma_{n-1}$, and $A_{n-1}^{0}$ and $A_{n-1}^{1}$ are the adjacency
matrices of the two instances of $G_{n-1}$. Explicitly opening the
products, we get
\begin{align*}
h^{T}A_{n}h & =\left(\begin{array}{cc}
h_{0}^{T} & h_{1}^{T}\end{array}\right)\left(\begin{array}{c}
A_{n-1}^{0}h_{0}+Ph_{1}\\
P^{T}h_{0}+A_{n-1}^{1}h_{1}
\end{array}\right)\\
 & =h_{0}^{T}A_{n-1}^{0}h_{0}+h_{0}^{T}Ph_{1}+h_{1}^{T}P^{T}h_{0}+h_{1}^{T}A_{n-1}^{1}h_{1}\\
 & \leq\left(n-3\right)\norm{h_{0}}_{2}^{2}+2h_{0}^{T}Ph_{1}+\left(n-3\right)\norm{h_{1}}_{2}^{2}\\
 & =\left(n-3\right)\norm h_{2}^{2}+2h_{0}^{T}Ph_{1}\leq\left(n-3\right)\norm h_{2}^{2}+2\norm{h_{0}}_{2}\norm{h_{1}}_{2}\\
 & \leq\left(n-3\right)\norm h_{2}^{2}+\norm{h_{0}}_{2}^{2}+\norm{h_{1}}_{2}^{2}=\left(n-2\right)\norm h_{2}^{2}.
\end{align*}
\end{proof}
\begin{proof}[Proof of Lemma \ref{lem:small_number_of_cycles}]
In the following, $C$ is a constant depending on $k$ whose value
may change from instance to instance. A set of edges $F\subseteq E\left(\boldsymbol{G}_{n}\right)$
is said to be ``finalized at generation $m$'' if for every edge
$\left\{ x,y\right\} \in F$, $\gamma\left(x,y\right)\leq m$, and
there exists at least one edge such that $\gamma\left(x,y\right)=m$.
For a given $u\in V_{n}$, let $w=N_{m}\left(u\right)$ be its $m$-neighbor,
and let $E_{m}\left(u\right)$ be event that there exists a cycle
of length no more than $k$ which contains the edge $\left\{ u,w\right\} $
and is finalized at generation $m$.

We will now bound the probability of the event $E_{m}\left(u\right)$.
Since $I_{m-1}\left(u\right)\neq I_{m-1}\left(w\right)$, i.e. $u$
and $w$ are found on different copies of $V_{m-1}$, in order for
a cycle of length $\leq k$ to exist, there must also be an $m$-generation
edge going from $I_{m-1}\left(w\right)$ back to $B_{<m}\left(u,k\right)$;
otherwise, any path starting with the edge $\left\{ u,w\right\} $
cannot reach $u$ again. In fact, this edge must be reachable from
$w$ in at most $k$ steps. Let $W$ be the set of all $z\in I_{m-1}\left(w\right)$
such that there exists a simple path $P=\left(x_{1},\ldots,x_{t}\right)$
with the following properties:
\begin{enumerate}
\item $P$ is a shortest path from $w$ to $z$, and $t\leq k$.
\item $\gamma\left(x_{i},x_{i+1}\right)\leq m$ for all $i=1,\ldots,t-1$.
\item $\gamma\left(x_{t-1},x_{t}\right)<m$.
\item $P$ does not contain the edge $\left\{ u,w\right\} $.
\end{enumerate}
In other words, $W$ is the set of all vertices in $I_{m-1}\left(w\right)$
which can be reached from $w$ by a path of at most $k$ edges of
generation at most $m$, and which can still send out an $m$-generation
edge without backtracking. If there are no edges from $W$ to $B_{<m}\left(u,k\right)$,
then there is no cycle of length $\leq k$ which contains $\left\{ u,w\right\} $
(see Figure \ref{fig:possible_cycle} for a graphical depiction). 

\begin{figure}[H]
\begin{centering}
\includegraphics[scale=0.5]{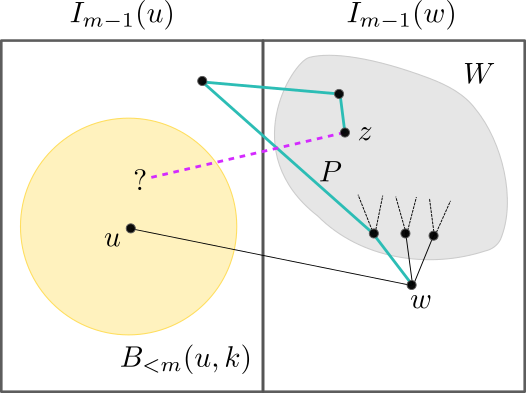}
\par\end{centering}
\caption{\label{fig:possible_cycle}There can be a cycle containing the edge
$\left\{ u,w\right\} $ only if there is an $m$-generation edge crossing
from some $z\in W$ to $B_{<m}\left(u,k\right)$. Since both $B_{<m}\left(u,k\right)$
and $W$ are small in comparison to $I_{m-1}\left(u\right)$, the
probability of this happening is small. }
\end{figure}

Given $z\in W$, the probability that $N_{m}\left(z\right)\in B_{<m}\left(u,k\right)$
depends only on the $m$-generation edges used in the path $P$. Since
$\sigma_{m-1}$ is uniform, we can bound this probability by 
\begin{align*}
\p\left[N_{m}\left(z\right)\in B_{<m}\left(u,k\right)\mid z\in W\right] & \leq\frac{\abs{B_{<m}\left(u,k\right)}}{\max\left\{ 1,2^{m-1}-k\right\} }\leq C\frac{m^{k+1}}{2^{m}}
\end{align*}
for some $C>0$ which depends on $k$ (we subtract $k$ in the denominator,
since in the worst case the path from $w$ to $z$ has at most $k$
$m$-generation edges from $I_{m-1}\left(w\right)$ to $I_{m-1}\left(u\right)\backslash B_{<m}\left(u,k\right)$).
Since there are at most $m^{k+1}$ vertices in $W$, taking the union
bound gives 
\[
\p\left[E_{m}\left(u\right)\right]\leq C\frac{m^{2k+2}}{2^{m}}.
\]
Letting $E_{m}=\union_{u\in B\left(v,2k\right)}E_{m}\left(u\right)$,
we then have 
\[
\p\left[E_{m}\right]\leq C\frac{m^{2k+2}}{2^{m}}n^{2k+1}.
\]
In particular, there exists a constant $C>0$ such that 
\[
\sum_{m>m_{0}}\p\left[E_{m}\right]\leq C2^{-m_{0}}m_{0}^{2k+2}n^{2k+1}.
\]
If a vertex $z\in B\left(v,k\right)$ is part of a cycle of length
at most $k$ which is finalized at generation $m$, then necessarily
there exists some $u\in B\left(v,2k\right)$ such that $E_{m}\left(u\right)$
holds. Thus, if $E_{m}^{c}$ holds for every $m>m_{0}$, then $z$
can only be contained in cycles of length at most $k$ which are finalized
at generation $\leq m_{0}$. The number of such cycles is bounded
by 
\[
\sum_{m=1}^{m_{0}}\sum_{i=1}^{k}m^{i}\leq\sum_{m=1}^{m_{0}}km^{k}\leq Cm_{0}^{k+1}
\]
for some constant $C>0$. The probability of $F_{v}$ is then lower
bounded by
\begin{align*}
\p\left[F_{v}\right] & \geq\p\left[\intersect_{m>m_{0}}E_{m}^{c}\right]\\
 & =1-\p\left[\union_{m>m_{0}}E_{m}\right]\\
 & \geq1-\sum_{m>m_{0}}\p\left[E_{m}\right]\\
 & \geq1-C2^{-m_{0}}m_{0}^{2k+2}n^{2k+1}
\end{align*}
as needed.
\end{proof}
\begin{proof}[Proof of Theorem \ref{thm:eigenvalues_follow_semicircle_law}]
We use the moment method. While the main technique is classical (see
e.g. \cite{mckay_eigenvalues_for_large_regular_graphs}), we write
the proof in full for completeness. 

Proving that $\mu_{n}$ converges weakly to $\circlaw$ in probability
means that for every continuous function $f:\r\to\r$, we have convergence
in probability of the expected value of $f$: 
\begin{equation}
\int_{\r}fd\mu_{n}\stackrel{P}{\to}\int_{\r}fd\circlaw\label{eq:meaning_of_convergence_in_probability}
\end{equation}
as $n\to\infty$. By the Weierstrass theorem, every continuous function
on a closed interval can be arbitrarily well-approximated by a finite-degree
polynomial. Since $\circlaw$ is supported on a bounded interval,
it suffices to show (\ref{eq:meaning_of_convergence_in_probability})
for functions of the form $f_{k}=x^{k}$, i.e. showing that the $k$-th
moments of $\mu_{n}$ converge the to $k$-th moments of $\circlaw$.
These moments are known, and are given by
\[
\int_{\r}x^{k}d\circlaw=\begin{cases}
C_{k/2} & k\text{ is even}\\
0 & k\text{ is odd},
\end{cases}
\]
where $C_{m}$ is the $m$-th Catalan number, and is equal to the
number of ordered rooted trees with $m$ edges. We will first show
that $\e\int_{\r}x^{k}d\mu_{n}\to\int_{\r}x^{k}d\circlaw$, and then
show that $\var\left(\int_{\r}x^{k}d\mu_{n}\right)\to0$; by Chebyshev's
inequality, this implies the desired convergence in probability. 

Since $\mu_{n}$ is just the empirical measure of the eigenvalues
of $A/\sqrt{n}$, we have 
\begin{align*}
\int_{\r}x^{k}d\mu_{n} & =\frac{1}{2^{n}}\sum_{i=1}^{2^{n}}\left(\frac{\lambda_{i}}{\sqrt{n}}\right)^{k}=\frac{1}{2^{n}}\tr\left(\frac{A}{\sqrt{n}}\right)^{k}\\
 & =\frac{1}{2^{n}}\frac{1}{n^{k/2}}\sum_{i_{1},\ldots,i_{k}=1}^{2^{n}}A_{i_{1}i_{2}}A_{i_{2}i_{3}}\cdot\ldots\cdot A_{i_{k-1}i_{k}}A_{i_{k}i_{1}}.
\end{align*}
For a fixed $i_{1}$, the sum $\sum_{i_{2},\ldots,i_{k}}A_{i_{1}i_{2}}\cdots A_{i_{k}i_{1}}$
is the number of walks of length $k$ in $\boldsymbol{G}_{n}$ that
start and end at the vertex $i_{1}$. Let $X_{v}\left(t\right)$ be
the simple random walk on $\boldsymbol{G}_{n}$ which starts at vertex
$v$. Then, since $\boldsymbol{G}_{n}$ is $n$-regular, the number
of simple random walks of length $k$ is $n^{k}$, and we have
\begin{align}
\int_{\r}x^{k}d\mu_{n} & =\frac{1}{2^{n}}n^{k/2}\sum_{i=1}^{2^{n}}\p\left[X_{i}\left(k\right)=i\right],\label{eq:moments_via_random_walks}
\end{align}
where the probability is over the randomness induced by the random
walk. Taking expectations over the measure induced by the permutations,
we thus have, for any $v\in V_{n}$, 
\begin{align*}
\e\int_{\r}x^{k}d\mu_{n} & =\frac{1}{2^{n}}n^{k/2}2^{n}\e\left[\p\left[X_{v}\left(k\right)=v\right]\right]\\
 & =n^{k/2}\e\left[\p\left[X_{v}\left(k\right)=v\right]\right].
\end{align*}

In the following, $C$ is a constant depending on $k$ whose value
may change from instance to instance. Let $m_{0}=8\left(k+1\right)\log_{2}n$.
By Lemma \ref{lem:small_number_of_cycles}, with probability greater
than $1-Cn^{-k}$, the event $F_{v}$ holds, i.e. every vertex in
$B\left(v,k\right)$ is contained in no more than $Cm_{0}^{k+1}$
cycles of length at most $k$. By conditioning on $F_{n}$, we have
\begin{align*}
\e\left[\p\left[X_{v}\left(k\right)=v\right]\right] & =\e\left[\p\left[X_{v}\left(k\right)=v\right]\mid F_{n}\right]\p\left[F_{n}\right]+\e\left[\p\left[X_{v}\left(k\right)=v\right]\mid F_{n}^{c}\right]\p\left[F_{n}^{c}\right].
\end{align*}
The second term on the right-hand side is bounded below by $0$ and
above by 
\[
\e\left[\p\left[X_{v}\left(k\right)=v\right]\mid F_{n}^{c}\right]\p\left[F_{n}^{c}\right]\leq\p\left[F_{n}^{c}\right]\leq Cn^{-k}=o\left(n^{-k/2}\right).
\]
Since $\p\left[F_{n}\right]=1-o\left(1\right)$, we then have 
\[
\e\left[\p\left[X_{v}\left(k\right)=v\right]\right]=\left(1+o\left(1\right)\right)\e\left[\p\left[X_{v}\left(k\right)=v\right]\mid F_{n}\right]+o\left(n^{-k/2}\right).
\]
To bound this term, we will count the number of random walks that
return to the origin. 

A step $\left(X_{v}\left(t\right),X_{v}\left(t+1\right)\right)$ is
said to be a \emph{forward step }if $d_{\boldsymbol{G}_{n}}\left(v,X_{v}\left(t\right)\right)<d_{\boldsymbol{G}_{n}}\left(v,X_{v}\left(t+1\right)\right)$,
and a \emph{backward step} if $d_{\boldsymbol{G}_{n}}\left(v,X_{v}\left(t\right)\right)\geq d_{\boldsymbol{G}_{n}}\left(v,X_{v}\left(t+1\right)\right)$.
By analyzing the combinatorics of forward and backward steps, it was
shown by McKay \cite[Lemma 2.1]{mckay_eigenvalues_for_large_regular_graphs}
that in an $n$-regular graph where every ball $B\left(v,k\right)$
has no cycles at all, 
\begin{equation}
\#\left\{ \text{Walks of length \ensuremath{k} which return to the origin}\right\} =\left(1+o\left(1\right)\right)n^{k/2}C_{k/2}.\label{eq:walks_in_acyclic_neighborhoods}
\end{equation}
We now show that under $F_{n}$, the number of walks in $\boldsymbol{G}_{n}$
is of the same magnitude. Let $\ell$ be the number of forward steps
of the walk $X_{v}\left(t\right)$ which are part of a cycle of length
no larger than $k$, and suppose that $X_{v}\left(k\right)=v$.

If $\ell=0$, then the walk must make $k/2$ forward steps and $k/2$
backward steps, since it returns to the origin. This means that $k$
must be even, and the walk traces out a rooted tree with $k/2$ edges.
Since the number of cycles with at most $k$ edges is no larger than
$C\left(\log n\right)^{k+1}$, there are at least least $n-C\left(\log n\right)^{k+1}-1$
choices for every forward step. By (\ref{eq:walks_in_acyclic_neighborhoods}),
the total number of walks with $\ell=0$ is then equal to 
\[
\left(1+o\left(1\right)\right)n^{k/2}C_{k/2}
\]
when $k$ is even, and $0$ when $k$ is odd. 

If $\ell>0$, then the walk makes $\ell$ forward steps which are
part of a cycle, and no more than $k/2-\ell$ forward steps which
are not part of a cycle. There are no more than $k$ backward steps,
and each such step has no more than $\left(C\left(\log n\right)^{k+1}+1\right)$
options. In total, the number of such walks is then bounded above
by 
\[
\left(1+o\left(1\right)\right)n^{k/2-\ell}\left(C\left(\log n\right)^{k+1}\right)^{k+\ell}=O\left(n^{k/2-\ell}\left(\log n\right)^{4k}\right).
\]
Altogether, since the total number of walks of length $k$ is $n^{k}$,
we have 
\begin{align*}
n^{k/2}\e\left[\p\left[X_{v}\left(k\right)=v\right]\right] & =n^{k/2}\left(1+o\left(1\right)\right)\frac{1}{n^{k}}\left(n^{k/2}C_{k/2}+\sum_{l=1}^{k/2}O\left(n^{k/2-\ell}\left(\log n\right)^{4k^{2}}\right)\right)\\
 & =\left(1+o\left(1\right)\right)C_{k/2}
\end{align*}
as needed.

All that is left is to show that the variance is small. By (\ref{eq:moments_via_random_walks}),
the second moment of $\int x^{k}d\mu_{n}$ is given by 

\begin{align*}
\e\left[\left(\int_{\r}x^{k}d\mu_{n}\right)^{2}\right] & =\e\left[\left(\frac{1}{2^{n}}n^{k/2}\sum_{i=1}^{2^{n}}\p\left[X_{i}\left(k\right)=i\right]\right)^{2}\right]\\
 & =\e\left[\frac{1}{2^{2n}}n^{k}\sum_{i,j=1}^{2^{n}}\p\left[X_{i}\left(k\right)=i\right]\p\left[X_{j}\left(k\right)=j\right]\right].
\end{align*}
Set $m_{0}=16\left(k+1\right)\log_{2}n$. Recall that for a vertex
$v\in V_{n}$, $F_{v}$ is the event that each vertex in $B\left(v,k\right)$
is contained in mo more than $C\left(m_{0}+1\right)^{k+1}$ cycles
of length no more than $k$. Denote $F_{i,j}=F_{i}\intersect F_{j}$.
By Lemma \ref{lem:small_number_of_cycles}, $\p\left[F_{i,j}\right]\geq1-2C2^{-m_{0}}m_{0}^{2k+2}n^{2k+1}$.
By the of total probability, we have 
\begin{align*}
\e\left[\left(\int_{\r}x^{k}d\mu_{n}\right)^{2}\right]= & \frac{1}{2^{2n}}\sum_{i,j=1}^{2^{n}}\e\left[n^{k}\p\left[X_{i}\left(k\right)=i\right]\p\left[X_{j}\left(k\right)=j\right]\mid F_{i,j}\right]\p\left[F_{i,j}\right]\\
 & \,\,\,+\frac{1}{2^{2n}}\sum_{i,j=1}^{2^{n}}\e\left[n^{k}\p\left[X_{i}\left(k\right)=i\right]\p\left[X_{j}\left(k\right)=j\right]\mid F_{i,j}^{c}\right]\p\left[F_{i,j}^{c}\right].
\end{align*}
The second term on the right-hand-side is bounded below by $0$ and
above, due to the choice of $m_{0}$, by $o\left(1\right)$. Thus
\[
\e\left[\left(\int_{\r}x^{k}d\mu_{n}\right)^{2}\right]=\left(1+o\left(1\right)\right)\frac{1}{2^{2n}}\sum_{i,j=1}^{2^{n}}\e\left[n^{k}\p\left[X_{i}\left(k\right)=i\right]\p\left[X_{j}\left(k\right)=j\right]\mid F_{i,j}\right]+o\left(1\right).
\]
Using the same path-counting argument as above, by (\ref{eq:walks_in_acyclic_neighborhoods})
we have that under $F_{i,j}$,
\[
\e n^{k}\p\left[X_{i}\left(k\right)=i\right]\p\left[X_{j}\left(k\right)=j\right]=\left(1+o\left(1\right)\right)C_{k/2}^{2},
\]
and taking the sum over all $i$ and $j$ shows that 
\[
\e\left[\left(\int_{\r}x^{k}d\mu_{n}\right)^{2}\right]=\left(1+o\left(1\right)\right)\e\int_{\r}x^{k}d\mu_{n},
\]
which implies that $\var\left(\int_{\r}x^{k}d\mu_{n}\right)\to0$.
\end{proof}

\subsection{Asymmetry\label{subsec:postponed_asymmetry_proofs}}

The proof of Theorem \ref{thm:asymmetry_of_g_n} relies on the following
lemma, whose proof we postpone to the end of this section.
\begin{lem}
\label{lem:no_partitions_with_matchings}There exists a constant $C>0$
such that the probability that there exists a decomposition of $V_{n}$
into two disjoint subsets other than $V_{n}^{0}\sqcup V_{n}^{1}$
such that the edges between them form a matching is smaller than $Cn2^{-n}$.
\end{lem}

\begin{proof}[Proof of Theorem \ref{thm:asymmetry_of_g_n}]
 Let $X_{n}$ be the number of automorphisms of $\boldsymbol{G}_{n}$.
We partition these permutations into three kinds:
\begin{enumerate}
\item Automorphisms of $\boldsymbol{G}_{n}$ that swap between $V_{n}^{0}$
and $V_{n}^{1}$. Let $W_{n}$ be the number of these automorphisms.
\item Automorphisms of $\boldsymbol{G}_{n}$ that preserve both $V_{n}^{0}$
and $V_{n}^{1}$. Let $Y_{n}$ be the number of these automorphisms.
Note that $Y_{n}\geq1$, since it always counts the trivial automorphism.
\item Automorphisms of $\boldsymbol{G}_{n}$ that replace a proper subset
$A_{0}\subseteq V_{n}^{0}$ with a proper subset $A_{1}\subseteq V_{n}^{1}$
of the same size (so that $V_{n}^{0}\backslash A_{0}$ and $V_{n}^{1}\backslash A_{1}$
stay inside $V_{n}^{0}$ and $V_{n}^{1}$, respectively). Let $Z_{n}$
be the number of these automorphisms.
\end{enumerate}
If $\varphi$ is a non-trivial automorphism of the third kind, then
the edges between $A_{0}$ and $V_{n}^{0}\backslash A_{0}$ form a
matching, and the edges between $A_{1}$ and $V_{n}^{1}\backslash A_{1}$
form a matching (since, e.g. if there is a vertex $v\in A_{0}$ connected
by more than one edge to $V_{n}^{0}\backslash A_{0}$, then $\varphi\left(v\right)$
will have more than one edge across the main cut). But then, letting
$A:=A_{0}\sqcup A_{1}$ and $B=V_{n}\backslash A$, we get that the
edges between $A$ and $B$ form a matching as well, giving a partition
$V_{n}=A\sqcup B$ with a matching between them. By Lemma \ref{lem:no_partitions_with_matchings},
the probability that such a matching exists (and therefore, that there
is a non-trivial automorphism swapping $A_{0}$ and $A_{1}$) is bounded
by $O\left(n2^{-n}\right)$. Thus, denoting by $F$ the event $F:=\left\{ \exists m\in\left[n/20,n-1\right]\text{ s.t. }Z_{m}>0\right\} $,
we have
\begin{equation}
\p\left[F\right]=O\left(n^{2}2^{-n/20}\right).\label{eq:bounding_probs_of_second_kind}
\end{equation}

We turn to bound $Y_{n},W_{n}$. For brevity, we abbreviate $\sigma:=\sigma_{n-1}$.
In the first two types, the values of an automorphism $\varphi$ on
$V_{n}^{0}$ determines the value of $\varphi$ on all $V_{n}$. Explicitly,
in the first case, for every $v\in V_{n-1}$, if we denote $\varphi(v,0)=(\varphi_{0}(v),1)$
and $\varphi\left(v,1\right)=\left(\varphi_{1}\left(v\right),0\right)$,
then we must have $\varphi_{0}\left(v\right)=\sigma\varphi_{1}\sigma(v)$.
For automorphisms of the second kind, we have similarly $\varphi_{0}=\sigma^{-1}\varphi_{1}\sigma$.
In both cases it must be that $\varphi_{0},\varphi_{1}\in\aut(\boldsymbol{G}_{n-1})$.
So in particular, $W_{n},Y_{n}\leq X_{n-1}$, and 
\begin{equation}
W_{n}+Y_{n}\leq2X_{n-1}.\label{eq:non_chaotic_autmorphs_cant_be_too_many}
\end{equation}

We first bound $\p\left[W_{n}\geq1|X_{n-1}\right]$. By Markov's inequality,
this is at most $\e[W_{n}|X_{n-1}].$ Write out $W_{n}=\sum_{\varphi_{0}}\sum_{\varphi_{1}}\one_{\varphi_{0}\in\aut\left(\boldsymbol{G}_{n-1}\right)}\cdot\one_{\varphi_{1}\in\aut\left(\boldsymbol{G}_{n-1}\right)}\cdot\one_{\varphi_{0}=\sigma\varphi_{1}\sigma},$
and in particular we have 
\[
\e\left[W_{n}|X_{n-1}\right]\leq X_{n-1}^{2}\max_{\varphi_{0},\varphi_{1}}\left\{ \p_{\sigma}\left[\varphi_{0}=\sigma\varphi_{1}\sigma\right]\mid\varphi_{0},\varphi_{1}\text{ bijections of }V_{n-1}\right\} .
\]
So we need to bound $\p_{\sigma}\left[\varphi_{0}=\sigma\varphi_{1}\sigma\right]$.
Denote $A_{0}=V_{n-1}$. For $v_{0}\in A_{0}$, let $E_{v_{0}}$ be
the event that $\varphi_{0}\left(v_{0}\right)=\sigma\varphi_{1}\sigma\left(v_{0}\right)$.
In order for $E_{v}$ to hold, we must have either i) $\varphi_{1}\sigma\left(v_{0}\right)=v_{0}$
or ii) $\varphi_{0}\left(v_{0}\right)=\sigma\varphi_{1}\sigma\left(v_{0}\right)$.
The probability of the first is $1/2^{n-1}$, while the probability
of the second given that $\varphi_{1}\sigma\left(v_{0}\right)\neq v_{0}$
is $\frac{1}{2^{n-1}-1}$ In particular $E_{v_{0}}$ holds with probability
no greater than $2/\left(2^{n-1}-1\right)$. Conditioned on $E_{v_{0}}$,
the permutation $\sigma$ is a uniform permutation over the set $A_{1}=A_{0}\setminus\{v_{0},\varphi_{1}\sigma\left(v_{0}\right)\}$,
with $\abs{A_{1}}\geq\abs{A_{0}}-2$. By iteratively conditioning
on $E_{v_{0}},E_{v_{0}},\ldots$, where $v_{i}\in A_{i}$, $i=0,\ldots,2^{n/3}-1$,
we have that $\p\left[\varphi_{0}=\sigma\varphi_{1}\sigma\right]\leq(\frac{2}{2^{n-1}-2^{1+(n/3)}-1})^{2^{n/3}}\leq2^{-(n-7)2^{n/3}}$.
Hence
\begin{equation}
\e\left[W_{n}\mid X_{n-1}\right]\leq2^{-(n-7)2^{n/3}}\cdot X_{n-1}^{2}.\label{eq:wn-bound}
\end{equation}

Next we bound $\p\left[Y_{n}>1|X_{n-1},F^{c}\right]$. Although the
equation $\varphi_{0}=\sigma^{-1}\varphi_{1}\sigma$ seems similar
to the analogous equation $\varphi_{0}=\sigma\varphi_{1}\sigma$ for
$W_{n}$, we shouldn't expect the same argument to hold, since (for
example) even if $X_{n-1}=1$, we expect $W_{n}=0$, whereas $Y_{n}\geq1$
always since it counts the identity. The problem lies with automorphisms
with small conjugacy classes. For a given $\varphi_{1}$ and uniformly
random $\sigma$, the element $\sigma^{-1}\varphi_{1}\sigma$ is a
uniform element in the conjugacy class of $\varphi_{1}$. The probability
$\p_{\sigma}\left[\varphi_{0}=\sigma^{-1}\varphi_{1}\sigma\right]$
is then bounded by one over the size of the conjugacy class of $\varphi_{0}$
(it is $0$ if $\varphi_{0}$ and $\varphi_{1}$ are not conjugate).
The following claim, whose proof is found at the end of the section,
shows that under $F^{c}$, these classes must be large.
\begin{claim}
\label{claim:conjugacies-are-large}Assume that $F^{c}$ occurs. Then
the conjugacy class for every $\mathrm{Id}\ne\varphi\in\aut\left(\boldsymbol{G}_{n-1}\right)$
has size at least $2^{\frac{1}{4}n2^{n/4}}$.
\end{claim}

As in the case of $W_{n}$, we have 
\begin{align*}
\p\left[Y_{n}>1|X_{n-1},F^{c}\right] & \leq\e\left[Y_{n}-1\mid X_{n-1},F^{c}\right]\\
 & \leq\e\left[X_{n-1}^{2}\mid F^{c}\right]\max_{\varphi_{0},\varphi_{1}\in\aut\left(\boldsymbol{G}_{n-1}\right)\backslash\left\{ \mathrm{Id}\right\} }\p_{\sigma}\left[\varphi_{0}=\sigma^{-1}\varphi_{1}\sigma\right]\\
 & \leq\e\left[X_{n-1}^{2}\mid F^{c}\right]2^{-\frac{1}{4}n2^{n/4}},
\end{align*}
where the last inequality is due to Claim \ref{claim:conjugacies-are-large},
since the maximum is taken over elements with a conjugacy class of
size at least $2^{\frac{1}{4}n2^{n/4}}$.

Finally, under $F^{c}$, the only possible automorphisms for $m\in\left[n/20,n-1\right]$
are of the first two kinds, and by (\ref{eq:non_chaotic_autmorphs_cant_be_too_many})
we have 
\begin{equation}
X_{n-1}\leq2^{n-n/20-1}X_{n/20}\leq2^{19n/20}\left(2^{n/20}\right)!\leq2^{n+\frac{1}{20}n2^{n/20}}.\label{eq:bounding_under_no_second_kind}
\end{equation}
Thus 
\begin{align*}
\p\left[X_{n}>1\right] & \leq\p\left[F\right]+\p\left[F^{c}\intersect\left\{ X_{n}>1\right\} \right]\\
 & \leq\p\left[F\right]+\p\left[Z_{n}>0\right]+\p\left[F^{c}\intersect\left\{ W_{n}>0\right\} \right]+\p\left[F^{c}\intersect\left\{ Y_{n}>1\right\} \right]\\
 & \leq\p\left[F\right]+\p\left[Z_{n}>0\right]+\p\left[W_{n}>0,X_{n-1}\leq2^{n+\frac{1}{20}n2^{n/20}},F^{c}\right]+\p\left[Y_{n}>1,X_{n-1}\leq2^{n+\frac{1}{20}n2^{n/20}},F^{c}\right]\\
 & \leq\p\left[F\right]+\p\left[Z_{n}>0\right]+\frac{2^{2n+\frac{1}{10}n2^{n/20}}}{2^{(n-7)2^{n/3}}}+\frac{2^{2n+\frac{1}{10}n2^{n/20}}}{2^{\frac{1}{4}n2^{n/4}}}=O\left(n^{2}2^{-n/20}\right)
\end{align*}
as needed.
\end{proof}
\begin{proof}[Proof of Lemma \ref{lem:no_partitions_with_matchings}]
With start with some preliminaries which will be of use later on
in the proof. Let $V_{n}=\boldsymbol{A}\sqcup\boldsymbol{B}$ be a
uniformly random partition of $V_{n}$ into two halves of equal size,
and let $V_{n}=\boldsymbol{A}'\sqcup\boldsymbol{B}'$ be a partition
where $\boldsymbol{A}'$ is a binomial random subset of $V_{n}$ with
success probability $1/2$. The difference between these two random
partitions can be quantified as follows: for any arbitrary set $\Sigma$
of equal-sized partitions of $V_{n}$, we have
\begin{equation}
\p\left[\left(\boldsymbol{A},\boldsymbol{B}\right)\in\Sigma\right]=\p\left[\left(\boldsymbol{A}',\boldsymbol{B}'\right)\in\Sigma\mid\abs{\boldsymbol{A}'}=2^{n-1}\right]\leq\frac{\p\left[\left(\boldsymbol{A}',\boldsymbol{B}'\right)\in\Sigma\right]}{\p\left[\abs{\boldsymbol{A}'}=2^{n-1}\right]}.\label{eq:relating_two_models}
\end{equation}
The denominator in the right-hand side can be approximated by the
de Moivre-Laplace limit theorem, which states that 
\begin{equation}
\p\left[\abs{\boldsymbol{A}'}=2^{n-1}\right]=\frac{1+o\left(1\right)}{\sqrt{\pi}}2^{\left(1-n\right)/2}.\label{eq:moiver_laplace}
\end{equation}
Note that $V_{n}$ contains a vertex-disjoint union of $2^{n-2}$
copies of $P_{3}$, the $3$ vertex path, and let $\Sigma$ be the
set of all equal-sized partitions which do not separate middle vertex
from the other two vertices of any of these paths. The probability
of splitting such a path is $\frac{1}{4}$. Thus $\p\left[\boldsymbol{A}'\sqcup\boldsymbol{B}'\in\Sigma\right]\leq e^{-\ln\frac{4}{3}\cdot2^{-n-6}}$,
and so by (\ref{eq:relating_two_models}) and (\ref{eq:moiver_laplace}),
we have
\begin{equation}
\p\left[\left(\boldsymbol{A},\boldsymbol{B}\right)\in\Sigma\right]\leq e^{-\ln\frac{4}{3}\cdot2^{-n-6}}\frac{\sqrt{\pi}}{1+o\left(1\right)}2^{\left(n-1\right)/2}\leq c\cdot e^{n-\ln\frac{4}{3}\cdot2^{-n-6}}\label{eq:equal_size_partition_cycle_prob}
\end{equation}
for some constant $c>0$. Let us now choose $C$ so large that the
lemma is true for all $n\leq n_{0}$ for some large enough $n_{0}$.
We proceed by induction on $n$. Assume that the lemma is true for
$\boldsymbol{G}_{n-1}$. For a decomposition $V_{n}=A\sqcup B$ other
than $V_{n}^{0}\sqcup V_{n}^{1}$, let $E_{n}\left(A,B\right)$ be
the event that the edges between $A$ and $B$ form a matching. Let
$p_{n}:=\p\left[\exists A,B\text{ s.t. }E_{n}\left(A,B\right)\right]$.
We will show that $p_{n-1}\leq C\left(n-1\right)2^{-\left(n-1\right)}$
implies $p_{n}\leq Cn2^{-n}$.

Fix a decomposition $V_{n}=A\sqcup B$ and set $A_{j}=V_{n}^{j}\intersect A$
and $B_{j}=V_{n}^{j}\intersect B$ for $j=0,1$, so that $A_{0}\sqcup B_{0}=V_{n}^{0}$
and $A_{1}\sqcup B_{1}=V_{n}^{1}$. We consider three cases.
\begin{enumerate}
\item If both cuts coincide with the cuts of the $\left(n-1\right)$-th
generation (i.e. $A_{0}\sqcup B_{0}=V_{n}^{00}\sqcup V_{n}^{10}$
and $A_{1}\sqcup B_{1}=V_{n}^{01}\sqcup V_{n}^{11}$), then for $A\sqcup B$
to induce a matching, all vertices of $A_{0}$ should send the edges
of the last generation to $A_{1}$, and all vertices of $B_{0}$ should
send the edges of the last generation to $B_{1}$. The probability
of this event is exactly ${2^{n-1} \choose 2^{n-2}}^{-1}$.
\item If both cuts differ from the $\left(n-1\right)$-th generation cuts,
then assume first that $A_{0}$ is empty. By connectivity of $\boldsymbol{G}_{n-1}$,
there exists an edge between some vertex $a\in A_{1}$ and a vertex
$b\in B_{1}$. Since $a$ also sends an edge to $B_{0}$ (induced
by $\sigma_{n-1}$), in this case the edges do not form a matching.
We can therefore assume that all of $A_{0},B_{0},A_{1},B_{1}$ are
non-empty. Assume without loss of generality that $\abs{B_{1}}\geq\abs{A_{1}}$.
Let $a\in A_{0}$ be a vertex that sends an edge to $B_{0}$ (again,
such a vertex exists by connectivity of $G_{n-1}$). Since $\abs{B_{1}}\geq\abs{A_{1}}$,
the probability that there are no edges from $a$ to $B_{1}$ induced
by $\sigma_{n-1}$ is at most $1/2$, so with probability greater
than $1/2$ we do not get a matching. We get that
\[
\p\left[\exists A,B\text{ s.t. }\left\{ A_{0}\sqcup B_{0}\neq V_{n}^{00}\sqcup V_{n}^{10},A_{1}\sqcup B_{1}\neq V_{n}^{01}\sqcup V_{n}^{11}\right\} \intersect E_{n}\left(A,B\right)\right]\leq\frac{1}{2}p_{n-1}.
\]
\item Finally, if, say, the cut $A_{0}\sqcup B_{0}$ coincides with the
respective $\left(n-1\right)$-th generation cut $V_{n}^{00}\sqcup V_{n}^{10}$,
and $A_{1}\sqcup B_{1}\neq V_{n}^{01}\sqcup V_{n}^{11}$, then $A_{1}\sqcup B_{1}$
should divide the set $V_{1}$ into halves; otherwise (say, if $\abs{A_{1}}>\abs{B_{1}}$),
$B_{0}$ sends at least one $n$-th generation edge to $A_{1}$, and
so there is a vertex in $B_{0}$ with at least two neighbors in $A$,
and we do not get a matching. Moreover, we may also claim that the
$n$-th generation edges form a matching between $A_{0}$ and $A_{1}$,
and between $B_{0}$ and $B_{1}$ (since there is a matching between
$A_{0}$ and $B_{0}$, $A_{0}$ cannot have an edge with $B_{1}$,
and $B_{0}$ cannot have an edge with $A_{1}$). Then the desired
probability is exactly the probability that the cut $A_{1}\sqcup B_{1}$
of $\boldsymbol{G}_{n-1}$ forms a matching. Since the ends of the
edges of the matching between $V_{n}^{0}$ and $V_{n}^{1}$ with first
vertices in sets $V_{n}^{00}$ and $V_{n}^{10}$ form a decomposition
of $\boldsymbol{G}_{n-1}$ into halves which is independent of $\boldsymbol{G}_{n-1}$
itself, by (\ref{eq:equal_size_partition_cycle_prob}), the latter
probability is bounded by $c\cdot2^{n-1-2^{n-7}}$.
\end{enumerate}
Putting all these together, we get 
\[
p_{n}\leq\frac{1}{2}p_{n-1}+{2^{n-1} \choose 2^{n-2}}^{-1}+c\cdot e^{n-\ln\frac{4}{3}\cdot2^{n-7}}<Cn2^{-n}
\]
for $n>n_{0}$ large enough.
\end{proof}
\begin{proof}[Proof of Claim \ref{claim:conjugacies-are-large}]
Assume that $F^{c}$ holds, i.e. every automorphism of $\boldsymbol{G}_{m}$
either swaps or preserves $V_{m}^{0},V_{m}^{1}$ for $m\in\left[n/20,n-1\right]$.
We first show that every non-identity $\varphi\in\aut(\boldsymbol{G}_{n-1})$
has at least $2^{19n/20}$ points that are not fixed. 

Let $m\in\left[n/20,n-1\right]$, and let $\psi\in\aut\left(\boldsymbol{G}_{m}\right)$.
Since $F^{c}$ holds, $\psi$ either swaps or preserves $V_{m}^{0}$
and $V_{m}^{1}$, and so can be represented by the pair $\left(\psi_{0},\psi_{1}\right)$
as above. If it swaps $V_{m}^{0}$ and $V_{m}^{1}$, then it has no
fixed points. Hence, if $\psi$ has any fixed points, it must be preserving,
and its fixed points are a union of the fixed points of $\psi_{0}$
and $\psi_{1}$. In this case $\psi_{0}$ and $\psi_{1}$ are conjugate,
so they have the same number of fixed points; in particular, the number
of fixed points (resp. non-fixed points) of $\psi$ is equal to twice
the number of fixed points (resp. non-fixed points) of $\psi_{0}$. 

Thus by induction, if $\varphi\in\aut\left(\boldsymbol{G}_{n-1}\right)$
has any fixed points, then the number of non-fixed points is equal
to $2^{n-1-n/20}$ times the number of non-fixed points of any of
its restrictions $\psi:=\varphi_{\mid V_{n}^{z}}$, where $z\in\left\{ 0,1\right\} ^{n-1-n/20}$.
If $\psi$ is the identity, then $\varphi$ is the identity also.
Otherwise, $\psi$ has at least $2$ non-fixed points, and so $\varphi$
has at least $2^{19n/20}$ non-fixed points on $\boldsymbol{G}_{n-1}$.

Next we get our bound for the size of the conjugacy class of $\varphi$.
Recall that we can express $\varphi$ as a composition of disjoint
cycles.
\begin{enumerate}
\item If $\varphi$ has a cycle of length $m\geq2^{2n/5}+1$, then the number
of conjugacy classes is bounded below by the number of conjugacy classes
where (say) $1$ is in such a cycle. The number of such permutations
is at least $\binom{2^{n-1}-1}{m-1}\cdot(m-1)!$ (since we need to
choose $m-1$ more elements, and then order them in a cycle together
with $1$). This is $(2^{n-1}-1)\cdot(2^{n-1}-2)\cdot...\cdot(2^{n-1}-m)\geq2^{(n-2)2^{2n/5}}$.
\item Otherwise, the maximal cycle length is at most $2^{2n/5}$. We have
at least $2^{19n/20}$ points which are not fixed, so there are at
least $r=2^{11n/20}$ cycles. The number of conjugacy classes is then
lower-bounded by the number of conjugacy classes where all the elements
$1,2,...,r$ are in distinct cycles, and $r+1,r+2,...,2r$ are in
the same distinct cycles . For every fixed choice of cycles for the
first $r$ elements, there are $r!$ ways to choose where to put $r+1,r+2,...,2r$.
This is at least $2^{11n/20}!\geq2^{\frac{1}{4}n2^{n/4}}$.
\end{enumerate}
\end{proof}

\section{Acknowledgments}

We thank Elad Tzalik for discussions about the diameter and open questions,
and Sahar Diskin for finding some of the references to existing literature.

\bibliographystyle{plain}
\bibliography{duplimatching_bibliography}

\end{document}